\documentclass[reqno,11pt]{amsart}

\usepackage{amssymb}
\usepackage{xcolor}

\usepackage{graphicx}
\usepackage[hidelinks]{hyperref}
\usepackage[shortlabels]{enumitem}
\usepackage{latexsym,array}
\usepackage{amsfonts}
\usepackage{shadow}
\usepackage{amsbsy}
\usepackage{dsfont}
\usepackage{mathtools}
\usepackage{mathrsfs}
\usepackage{chngcntr}
\usepackage{cleveref}
\usepackage[notref, notcite, final]{showkeys}
\usepackage{comment}
\usepackage[latin1]{inputenc}
\usepackage{doi}

\newtheorem{theorem}{Theorem}[section]
\newtheorem{lemma}[theorem]{Lemma}
\newtheorem{corollary}[theorem]{Corollary}

\theoremstyle{definition}
\newtheorem{remark}[theorem]{{\bf Remark}}
\newtheorem{definition}[theorem]{Definition}

\newcommand{\hh}{\mathbb{H}}

\newcommand{\rr}{\mathbb{R}}

\newcommand{\boundOP}{\mathcal{B}}

\newcommand{\id}{\mathcal{I}}

\renewcommand{\Re}{\mathrm{Re}}

\newcommand{\BN}{\mathbb{N}}

\newcommand{\BD}{\mathbb{D}}

\newcommand{\BC}{\mathbb{C}}
\newcommand{\BI}{\mathbb{I}}
\newcommand{\BJ}{\mathbb{J}}
\newcommand{\BH}{\mathbb{H}}
\newcommand{\BS}{\mathbb{S}}

\newcommand{\cB}{\mathcal{B}}	
\newcommand{\cH}{\mathcal{H}}

\newcommand{\sN}{\mathscr{N}}
\newcommand{\sM}{\mathscr{M}}
\newcommand{\cM}{\mathcal{M}}
\newcommand{\cL}{\mathcal{L}}
\newcommand{\cC}{\mathcal{C}}
\newcommand{\cJ}{\mathcal{J}}

\newcommand{\ov}{\overline}

\newcommand{\llangle}{\langle \kern -0.2em \langle}	
\newcommand{\rrangle}{\rangle \kern -0.2em \rangle}

\crefname{enumi}{}{}
\crefname{enumii}{}{}

\title[Perturbation of normal quaternionic operators]
{Perturbation of normal quaternionic operators}
 \oddsidemargin
0.2in \evensidemargin 0.2in \topmargin -0.5in \textwidth
15.5truecm \textheight 23truecm

\author[P. Cerejeiras]{P. Cerejeiras}
\address{(PC) CIDMA, Departamento de Matem\'atica\\
Universidade de Aveiro\\
Campus de Santiago\\
P-3810-193, Aveiro\\
Portugal}
\email{pceres@ua.pt}
\email{}

\author[F. Colombo]{F. Colombo}
\address{(FC) Politecnico di
Milano\\Dipartimento di Matematica\\Via E. Bonardi, 9\\20133 Milano\\Italy}
\email{irene.sabadini@polimi.it}

\author[U. K\"ahler]{U. K\"ahler}
\address{(UK) CIDMA, Departamento de Matem\'atica\\
Universidade de Aveiro\\
Campus de Santiago\\
P-3810-193, Aveiro\\
Portugal}
\email{ukaehler@ua.pt}

\author[I. Sabadini]{I. Sabadini}
\address{(IS) Politecnico di
Milano\\Dipartimento di Matematica\\Via E. Bonardi, 9\\20133 Milano\\Italy}
\email{fabrizio.colombo@polimi.it}

\begin{document}

\begin{abstract}
The theory of quaternionic operators has applications in several  different fields such as  quantum mechanics, fractional evolution problems, and quaternionic Schur analysis, just to name a few. The main difference between complex and quaternionic operator theory is based on the definition of spectrum. In fact, in quaternionic operator theory the classical  notion of resolvent operator and the one of spectrum need to be replaced by the two $S$-resolvent operators and the $S$-spectrum. This is a consequence of the non-commutativity of the quaternionic setting. Indeed, the $S$-spectrum  of a quaternionic linear operator $T$ is given by the non invertibility of a second order operator. This presents new challenges which makes our approach to perturbation theory of quaternionic operators different from the classical case. In this paper we study the problem of perturbation of a quaternionic normal operator in a Hilbert space by making use of the concepts of $S$-spectrum and of slice hyperholomorphicity of the $S$-resolvent operators. For this new setting we prove results on the perturbation of quaternionic normal operators by operators belonging to a Schatten class and give conditions which guarantee the existence of a nontrivial hyperinvariant subspace of a quaternionic linear operator.
\end{abstract}
\maketitle

\vskip 1cm
\par\noindent
 AMS Classification: 47A10, 47A60.
\par\noindent
\noindent {\em Key words}: Perturbation quaternionic operators, $S$-spectrum, $S$-resolvent operators, Slice hyperholomorphic functions.
\vskip 1cm

\section{Introduction}

The spectral theory for quaternionic linear operators and, as a particular case, for vector operators,
has been an open problem for a long time because the notion of spectrum of a quaternionic linear operator was unclear. Indeed, the notions of left or right spectrum of a quaternionic linear operator are not suitable to develop a full theory. This situation changed a decade ago when the new notion of $S$-spectrum was introduced, see the book \cite{CSS}.
\\
\\
There are several reasons to study quaternionic spectral theory and below we mention some of them. First of all there is an interest coming from the study of partial differential equations (or more generally of pseudo-differential operators) over non-commutative structures.  Large attention has been given to the case of nilpotent Lie groups which has been studied since the 1970s (see~\cite{Ruzhansky}, ~\cite{Stein}). This is due to the fact that the corresponding Baker-Campbell-Hausdorff formula is then finite (that is to say the higher order commutators are zero after a finite order) which allows for an easier theory. Even more challenging are partial differential equations and pseudo-differential operators over other structures like quaternions where the Baker-Campbell-Hausdorff formula only exhibits a periodic nature (higher order commutators are equal to lower order commutators after a finite order) since these cases are obviously more complicated. The case of quaternions is even more important since they are closely linked to symmetries in the phase space and, therefore, the corresponding PDE's have also close connections with both time-frequency analysis and quantum mechanics. In the celebrated paper \cite{BvN},  Birkhoff and von Neumann
showed that quantum mechanics can be written only in the real, complex or in the quaternionic setting. This fact stimulated a numbers of works among which we mention \cite{adler, 12, 14, 21}.
Recently, see  \cite{ack,acks3}, the spectral theorem based on the $S$-spectrum for quaternionic normal operators was proved. This provides the grounds to study quantum mechanics in the quaternionic setting.
\\
Another question which was recently solved was to find a quaternionic analogue of the Riesz-Dunford functional calculus of the complex setting. This calculus can be naturally extended to quaternionic operators using the theory of slice
 hyperholomorphic functions, the $S$-spectrum and the $S$-resolvent operators which are the crucial objects to properly define the quaternionic functional calculus, also called $S$-functional calculus, see the books \cite{CSS,CSS2016}, and  \cite{acgs,DA}. This calculus allows to study the theory of quaternionic evolution operators which was developed in \cite{perturbation,FUCGEN,evolution}.
 We also point out that the $S$-resolvent operators naturally appear in the realization of quaternionic Schur functions
 which has allowed a rapid development of Schur analysis in the slice hyperholomorphic setting, see the book \cite{acsbook}.
\\
\\
More recently, it turned out that quaternionic spectral theory  is also a useful tool to study new classes of
fractional evolution problems.
In fact, using the quaternionic version of the $H^\infty$ functional calculus, one can define fractional powers of vector operators and obtain a new approach to fractional diffusion processes, see
\cite{Hinfty,FJTAMS,OurNewPaper} for more details.
\\
\\
The above facts provide sound motivations to consider the perturbation theory of quaternionic linear operators, whose investigation is well alive even in the classic complex, e.g. in \cite{Kato, KM, KMR}.
\\
\\
To understand the additional difficulties compared to the classical case which arise in a non-commutative setting, we now discuss some of the main differences between classical spectral theory and the spectral theory based on the $S$-spectrum.
Let us begin by recalling that, given a bounded complex linear operator $A$ acting on a complex Banach space $X$, its spectrum is defined as
by
$$
\sigma(A)=\{\lambda \in \mathbb{C}\ :\ \lambda \mathcal{I}-A\ \   \text{is not invertible}\}.
$$
For $\lambda$ in the resolvent set $\rho(A):=\mathbb{C}\setminus \sigma(A)$, the resolvent operator $ (\lambda \mathcal{I}-A)^{-1}$ is a holomorphic function with values in the Banach space $\mathcal{B}(X)$ of all bounded linear operators  on $X$ endowed with the natural norm. Now let us consider a bounded linear operator defined on a two sided quaternionic Banach space $V$; given a quaternionic operator $T$ one has to specify on which side the linearity is considered. In this paper we consider right linearity but, for the sake of simplicity, in this introduction we will simply write quaternionic linear operator without specifying the type of linearity where not needed. The operator $s \mathcal{I}-T$ acts on a vector $v\in V$ as $sv-Tv$ while $\mathcal{I}s-T$ acts on a vector $v\in V$ as $vs-Tv$. The first operator is right linear over $\mathbb H$ while the second is not. We note also that the first operator, though linear, does not seem to have any physical meaning, while the second gives the notion of right eigenvalues which is widely used in Physics and in linear algebra over non-commutative structures.
   Moreover, the inverse of both the operators above is not associated to any notion of hyperholomorphy. For these reasons,  in the quaternionic setting the appropriate notion of spectrum is the one of $S$-spectrum, which is defined by a second order operator.
Specifically, the $S$-spectrum of a quaternionic linear operator $T$ is defined as
$$
\sigma_S(T)=\{s\in \mathbb{H} \ :\ T^2-2{\rm Re}(s)T+|s|^2\mathcal{I}\ \ \text{is not invertible}\}
$$
where $\mathbb{H}$ denotes the algebra of quaternions, ${\rm Re}(s)$ is the real part of the quaternion $s$, and $|s|^2$ is the square of its
Euclidean norm. It is important to note that the point $S$-spectrum coincides with the set of right eigenvalues, see \cite{spectrum,GMP}, thus the operator $Q_s(T):=(T^2-2{\rm Re}(s)T+|s|^2\mathcal{I})^{-1}$, called pseudo-resolvent operator, is the linear operator associated with the notion of right eigenvalues.
The operator $(T^2-2{\rm Re}(s)T+|s|^2\mathcal{I})^{-1}$ is defined on the $S$-resolvent set $\rho_S(T):=\mathbb{H}\setminus \sigma_S(T)$ and it is a continuous function with values in the space $\mathcal{B}(V)$ of all bounded quaternionic linear operators, but it is not hyperholomorphic with respect to any known notion of hyperholomorphicity.
To define the analogue of the resolvent operator $(\lambda \mathcal{I}-A)^{-1}$ with some analyticity properties, denote by $\overline{s}$ the conjugate of the quaternion $s$, and we define the $S$-resolvent operators as
$$
S^{-1}_L(s,T):=-Q_s(T)( T-\overline{s}\mathcal{I})\ \ \ \text{and}\ \ \ S^{-1}_R(s,T):=-(T-\overline{s}\mathcal{I})Q_s(T).
$$
These operators defined on $\rho_S(T)$ are right and left slice hyperholomorphic operator-valued functions, respectively, where the notions of left and right  slice-hyperholomorphic functions will be defined in the sequel.
 Thus in the quaternionic setting there are two resolvent operators $S^{-1}_L(s,T)$ and $S^{-1}_R(s,T)$,
 and, moreover, the $S$-resolvent equation involves both the $S$-resolvent operators, see Section 2 for more details.
 Using the notion of slice hyperholomorphic functions we can define the analog of the Riesz-Dunford functional calculus for quaternionic operators
and in a natural way we can define the Riesz-projectors, see \cite{acgs}.
\\
\\
 We also want to point out that there exists other approaches to functional calculi in higher dimensions. In a series of papers, see for example \cite{jefferies,jmc,jmcpw,mcp},  McIntosh and co-authors
introduced and studied the functional calculus for $n$-tuples of operators using the more classical theory of
 monogenic functions. In this theory, one introduces a different notion of spectrum based on the Cauchy integral formula for monogenic functions. This theory, however, lacks the appropriate tools for the study of perturbations of normal operators such as the lack of a spectral theorem.
 \\
 \\
 The literature contains a great amount of works on invariant subspaces  of  operators in a Hilbert space; without claiming completeness, we mention as examples the works of  Livsic  \cite{Livsic}, Brodskii \cite{Brodskii}, Sz.Nagy-Foias and co-authors \cite{SzNFBK}, Gohberg and Krein  \cite{GK1, GK2} and the references therein.
   In this paper, we consider the problem of the perturbation of normal operators in a Hilbert space and the existence of (hyper-)invariant subspaces for quaternionic normal operators. The classic results in the complex case can be found in the book \cite{RR1973} by  Radjavi and Rosenthal.
The knowledge of invariant subspaces gives information on the structure of operators, however they do not always exist: there exist  bounded linear operators on a complex inner-product space without a non-trivial invariant subspace.
    We are going to study compact perturbations of normal operators on a quaternionic Hilbert space whose spectrum lies on a smooth Jordan arc for specific 2D-subspaces, later called slices.
 From these perturbation results one can deduce, under suitable assumptions, the existence of invariant subspaces.
 We also discuss the existence of hyperinvariant subspaces, which are related with the structure of the so-called commutant of $T$, namely the set of operators commuting with $T$.
  We work in a class of vector-valued slice hyperholomorphic functions that have a slice hyperholomorphic continuation across arcs contained in the $S$-spectrum of the operators intersected with a complex plane. As we shall see this is not reductive, provided the symmetry properties on the $S$-spectrum.
  \\
  It is necessary to point out that the non-commutative setting of quaternions involves several challenges from the technical side. If one looks at the classic proofs in \cite{RR1973} one can easily see that they are heavily dependent on the commutativity of the underlying complex field. For example, in the complex case, given a linear operator $A$, any linear operator $B$ commuting with $A$ also commutes with the resolvent $(\lambda \mathcal{I}-A)^{-1}$. In the quaternionic case, a right linear operator $B$ commuting with a given right linear operator $T$ does not commute, in general, with the $S$-resolvent operators because it does not commute with the quaternionic variable. It does commute with the pseudo-resolvent, since it has only real coefficients, but as we have discussed, this operator does not have any analyticity property. Additionally, in the quaternionic setting one has to face the fact that the algebraic inverse of the (non-linear) operator $T-\mathcal{I}s$, the operator which gives the spectrum (the pseudo-resolvent $Q_s(T)$) and the two $S$-resolvent operators correspond to four different operators. As a matter of fact, the algebraic inverse plays no role.  The pseudo-resolvent and the two $S$-resolvent operators are all required for the proofs of the various results. Furthermore, the two $S$-resolvent operators cannot be simply used in an arbitrary order. This also allows us to demonstrate the properties which are really required in the quaternionic setting. \\
\\
The plan of the paper is as follows.
In Section 2 we introduce the splitting of the $S$-spectrum in approximate point $S$-spectrum and compression $S$-spectrum and we show some related results, moreover we give a quick overview of the on the $S$-functional calculus.
Section 3 contains some results related with quaternionic normal operators.
In Section 4 we show some results on the Schatten class of quaternionic normal operators.
In Section 5 we state and prove our main results on the perturbation of normal operators and some consequences.
More specifically, we prove results which guarantee the existence of a nontrivial (hyper-)invariant subspace of a quaternionic linear operator $T$ and we discuss some consequences.

\section{Preliminary results on quaternionic bounded operators}

This section contains, besides some preliminaries, new results on the properties of the splitting of the $S$-spectrum of a quaternionic linear operator
of $T$ in terms of the approximate point $S$-spectrum $\Pi_{S}(T)$  and the compression $S$-spectrum $\Gamma_S(T)$ of $T$. Finally, we recall the $S$-functional calculus.
\\
\\
We denote by $\hh$ the algebra of quaternions.
The imaginary units  in $\hh$ are denoted by $e_1$, $e_2$, $e_3$, they satisfy the relations
$ e_1^2=e_2^2=e_3^2=-1$, $e_1e_2 =-e_2e_1 =e_3$, $e_2e_3 =-e_3e_2 =e_1$, $e_3e_1 =-e_1e_3 =e_2$
 and an element in $\hh$ is of the form $q=x_0+e_1x_1+e_2x_2+e_3x_3$, for $x_\ell\in\rr$.
The real part, the imaginary part and the modulus of a quaternion are defined as
$ {\rm Re}(q)=x_0$, ${\rm Im}(q)=e_1x_1+e_2x_2+e_3x_3$,
$|q|^2=x_0^2+x_1^2+x_2^2+x_3^2$, respectively.
The conjugate of the quaternion $q=x_0+e_1x_1+e_2x_2+e_3x_3$ is defined by
$$
\bar q={\rm Re }(q)-{\rm Im }(q)=x_0-e_1x_1-e_2x_2-e_3x_3.
$$
Let us denote by $\mathbb{S}$ the unit sphere of purely imaginary $\mathbb{S}$
quaternions, i.e.
$$
\mathbb{S}=\{q=e_1x_1+e_2x_2+e_3x_3\ {\rm such \ that}\
x_1^2+x_2^2+x_3^2=1\}.
$$
Given a non-real quaternion $q=x_0+{\rm Im} (q)=x_0+\BJ |{\rm Im} (q)|$, $\BJ={\rm Im} (q)/|{\rm Im} (q)|\in\mathbb{S}$, we can associate to it the 2-dimensional sphere defined by
$$
[q]=\{x_0+\BJ  |{\rm Im} (q)| \ : \ \BJ \in\mathbb{S}\}.
$$

We also need the notion of slice-hyperholomorphicity which will replace the notion of analyticity in the functional calculus. Let us start with the notion of axially symmetric domains.

\begin{definition} \label{Def:2.1}
Let $U \subseteq \mathbb{H}$ be an open set. We say that $U$ is
\textnormal{axially symmetric} if, for all $u+\BJ v \in U$, the whole
2-sphere $[u+\BJ v]$ is contained in $U$.
\end{definition}

The above notion allows us to introduce the notion of slice-hyperholomorphicity for functions defined over axially symmetric domains.

\begin{definition}[Slice hyperholomorphic functions]  \label{Def:2.2}
 Let $U\subseteq\hh$ be an axially symmetric open set
 and let $\mathcal{U}\subseteq\mathbb{R}\times\rr$ be such that $q=u+\BJ v\in U$ for all $(u,v)\in\mathcal{U}$.
We say that a functions on $U$ of the form
$$f(q)=\alpha(u,v)+\BJ \beta(u,v)$$
is left slice hyperholomorphic if
 $\alpha$, $\beta$ are $\mathbb{H}$-valued differentiable functions such that
 $$
 \alpha(u,v)=\alpha(u,-v), \ \ \ \beta(u,v)=-\beta(u,-v)\ \ \ {\rm for \ all}\ \  (u,v)\in \mathcal{U}
 $$
 and if $\alpha$ and $\beta$ satisfy the Cauchy-Riemann system
 $$
\partial_u \alpha-\partial_v\beta=0,\ \ \ \ \
\partial_v\alpha+\partial_u \beta=0.
$$
When $f$ is of the form
$$
f(q)=\alpha(u,v)+\beta(u,v) \BJ
$$
with the above properties for $\alpha$ and $\beta$ we say that $f$ is a right slice hyperholomorphic function on $U$.
The set of left (resp. right) slice hyperholomorphic functions on $U$ will be denoted by $\mathcal{SH}_L(U)$ (resp. $\mathcal{SH}_R(U)$).
Slice hyperholomorphic functions on $U$ such that $\alpha(u,v)$ and $\beta(u,v)$ are real-valued are called intrinsic and the corresponding set is denoted by $\mathcal{N}(U)$.
\end{definition}

\begin{definition}\label{starprod:left}
Let $U\subseteq\mathbb{H}$ be an axially symmetric open set and let
$f,g:  U\to \mathbb{H}$ be left slice hyperholomorphic functions. Let
$f(x+\mathbb{J}y)=\alpha(x,y)+\mathbb{J}\beta(x,y)$,
$g(x+\mathbb{J}y)=\gamma(x,y)+\mathbb{J}\delta(x,y)$. Then we define $\star_l$-product as
\begin{equation}\label{opleft}
(f\star_l g)(x+\mathbb{J} y):= (\alpha\gamma -\beta \delta)(x,y)+
\mathbb{J}(\alpha\delta +\beta \gamma )(x,y).
\end{equation}
\end{definition}
It can be easily verified that,
by its construction, the function $f\star_l g$ is left slice
hyperholomorphic. A similar multiplication can be defined in the case of right slice hyperholomorphic functions and it is denoted by $\star_r$, according to the position of $\mathbb J$. When it is not needed to distinguish the two cases, we will simply write $\star$.
\begin{remark}\label{remark:series}{\rm
Let us consider the case in which $U$ is a ball with center at a real point (let us assume at
the origin for simplicity). Then it  is immediate to verify (with the techniques used in the complex case), that $f$ is slice hyperholomorphic if and only if it admits in $U$ a converging power series
expansion  $f(q)=\sum_{n= 0}^\infty
q^n a_n$, $a_n\in\mathbb{H}$. If $g(q)=\sum_{n= 0}^\infty q^n b_n$, with $b_n\in \mathbb{H}$ for all
$n$, then
\[
(f\star_l g)(q):=\sum_{n= 0}^\infty q^n (\sum_{r=0}^n
a_rb_{n-r}).
\]
Thus, in this case, the notion of $\star_l$-product coincides with the classical notion of product of series with coefficients in a ring.}
\end{remark}

\begin{remark}\label{star}{\rm
The above notions of  slice hyperholomorphicity and $\star$-multiplication can be extended to operator-valued functions, see \cite{acsbook}. We can also define a notion of inverse of a function with respect to the $\star$-product, see e.g. \cite{acsbook}, but since we do not need it in its full generality, we introduce it just for the case we will need, namely the $\star$-inverse of the function $f(q)=q-s$ which is both left and right hyperholomorphic in $q$. We have
\[
\begin{split}
(q-s)^{-\star_l}&=(q^2- 2{\rm Re}(s) q+|s|^2)^{-1}(q-\bar s)\\
(q-s)^{-\star_r}&=(q-\bar s)(q^2- 2{\rm Re}(s) q+|s|^2)^{-1}.
\end{split}
\]
We note that $g(s)=q-s$ is left and right slice hyperholomorphic in $s$ and we can construct its left and right $\star$-inverses in the variable $s$. These inverses are related to the inverses $f^{-\star_l}$, $f^{-\star_r}$ computed in $q$ as follows, since, see e.g. \cite{CSS}, the following identities hold:
\[
\begin{split}
f^{-\star_l}(q)=(q^2 -2 {\rm Re} (s) q+|s|^2)^{-1}(q-\bar s)&=-(s-\bar q)(s^2-2{\rm
Re}(q) s+|q|^2)^{-1}=-g^{-\star_r}(s)\\
f^{-\star_r}(q)=(q-\bar s)(q^2-2{\rm Re}(s) q+|s|^2)^{-1}&= - (s^2-2{\rm Re}(q)s+|q|^2)^{-1}(s-\bar q)=-g^{-\star_l}(s) .
\end{split}
\]
}
\end{remark}
Slice hyperholomorphic functions are those functions for which the $S$-functional calculus can be defined, as we will see in the sequel.

Now, let $V$ be a right vector space on $\mathbb{H}$, and consider a right linear operator $T$ on $V$, i.e.
an operator such that
$$
 T(u+v)=T(u)+T(v),\ \ \ \
T(us)=T(u)s, \ \ {\rm for\  all}\   s\in\mathbb{H},\  u,v\in V.
$$
In the sequel we will denote by $\mathcal{B}(V)$ the quaternionic Banach space of all right linear bounded operators endowed with the natural norm. Furthermore, a vector space $V$ which is a quaternionic Hilbert space will be denoted by $\mathcal H$ and the Banach algebra of right linear quaternionic bounded operators on $\cH$ will be denoted by $\cB(\cH)$.

\begin{definition} \label{Def:4.4}
Let $T\in \mathcal{B}(\mathcal{H})$. The adjoint $T^*$ of $T$ is the unique operator that satisfies
\[ \langle T^* x, y\rangle = \langle x, Ty\rangle\quad\forall x,y\in\mathcal{H}.\]
Moreover,
 for  $T\in\mathcal{B}(\mathcal{H})$, we say that the operator is
 selfadjoint, if $T^* = T$; anti-selfadjoint, if $T^* = - T$; normal if $T^*T = TT^*$;
 unitary if $T^* = T^{-1}$;
 positive if $\langle x, Tx\rangle \geq 0$ for all $x\in\mathcal{H}$.
 \end{definition}

Since the standard generalization of the notion of spectrum as well as the question of solvability of the equation $Tv-vs=0$ lead to discussions of invertibility of a nonlinear operator, we are going to use the following notion of the S-spectrum, which reduces the question to the invertibility of a second order scalar (with respect to the underlying algebra) operator.

\begin{definition} \label{Def:2.3}
Let $T\in\mathcal{B}(V)$. We define the {\em S-spectrum } of $T$ as
$$
\sigma_S(T) = \{s\in\mathbb{H}: T^2 - 2\Re(s)T + |s|^2\id \text{ is not invertible}\}
$$
and we define the {\em S-resolvent set }  of $T$ as
$$
\rho_S(T) = \mathbb{H}\setminus\sigma_S(T).
$$
\end{definition}
Hereby, the second order operator
$$
Q_s(T):=(T^2 - 2\Re(s)T + |s|^2\id)^{-1},\ \ \ s\in \rho_S(T),
$$
 will be called the pseudo-resolvent operator.

While we have the natural notion of the pseudo-resolvent $Q_s(T)$ this does not give a good replacement of the classic resolvent operator $(A-\lambda\mathcal{I})^{-1}$ alone since it originates from a second order operator. In fact, for the actual study of the operator $T$ we also need the notion of the $S$-resolvent operator which can be defined in a left- and a right-form. As will be clear in the sequel only the interplay of all three operators will provide an adequate replacement of the classic resolvent operator.

\begin{definition}\label{Def:3.7}
Let $T\in\mathcal{B}(V)$. For $s\in\rho_S(T)$, we define the {\em left S-resolvent operator} as
$$
S_L^{-1}(s,T) = -(T^2-2\Re(s)T + |s|^2\mathcal{I})^{-1}(T-\overline{s}\,\mathcal{I}),
$$
and the {\em right S-resolvent operator} as
$$
S_R^{-1}(s,T) = -(T-\overline{s}\id)(T^2-2\Re(s)T+|s|^2\mathcal{I})^{-1}.
$$
\end{definition}
It is easy to show that the S-resolvent operators are slice hyperholomorphic operator-valued functions.

\begin{theorem}\label{ResolventRegular}\cite{acsbook}
Let $T\in\mathcal{B}(V)$.
\begin{enumerate}
\item[(i)]
The left S-resolvent operator $S_L^{-1}(s,T)$ is a $\boundOP(V)$-valued right-slice hyperholomorphic function of the variable $s$ on $\rho_S(T)$.
\item[(ii)]
 The right S-resolvent operator $S_R^{-1}(s,T)$ is a $\mathcal{B}(V)$-valued left-slice hyperholomorphic function of the variable $s$ on $\rho_S(T)$.
 \end{enumerate}
\end{theorem}

Furthermore, they give also rise to a resolvent equation involving both resolvent operators. Hereby, one has to be careful that due to the non-commutativity the resolvent operators can only be used in a fixed order.

\begin{theorem}\label{SREQ}\cite{acgs}
Let $T\in \mathcal{B}(V)$ and let $s,p\in \rho_S(T)$. Then the equation
\[
\begin{split}\label{SREQ1}
S_R^{-1}(s,T)S_L^{-1}(p,T)&=\left[\left(S_R^{-1}(s,T)-S_L^{-1}(p,T)\right)p-\overline{s}\left(S_R^{-1}(s,T)-S_L^{-1}(p,T)\right)\right]
\\
&\cdot(p^2-2\Re(s)p+|s|^2)^{-1}
\end{split}
\]
holds true. Equivalently, it can also be written as
\[
\begin{split}\label{SREQ2}
S_R^{-1}(s,T)S_L^{-1}(p,T)&=(s^2-2\Re[p]s+|p|^2)^{-1}
\\
&
\cdot\left[\left(S_L^{-1}(p,T) - S_R^{-1}(s,T)\right)\overline{p}-s\left(S_L^{-1}(p,T) - S_R^{-1}(s,T)\right)
 \right].
\end{split}
\]
\end{theorem}

As in the complex case, it is possible to define some splitting of the $S$-spectrum and to this end we recall the following well known theorem, whose proof is the same as in the complex case:
\begin{theorem}\label{Th:3.1}
A quaternionic bounded linear operator $A$ that satisfies the two conditions:
\begin{enumerate}
\item[(i)] There exists $K> 0$ such that $\|Av\|\geq K\|v\|$ for $v\in D(A)$ ($A$ is bounded from below)
\item[(ii)] the range of $A$ is dense
\end{enumerate}
is invertible.
\end{theorem}
The following definition for the splitting of the spectrum is based on the previous theorem on the invertibility of linear operators:
 \begin{definition} \label{Def:3.2}
 The point $S$-spectrum of $T$, denoted by $\Pi_{0,S}(T)$, is defined as
 $$
 \Pi_{0,S}(T)=\{   s \in \mathbb{H} \ : \ T^2 - 2\Re(s)T + |s|^2\mathcal{I}  \text{ is not one-to-one}\}.
 $$
The approximate point $S$-spectrum of $T$, denoted by $\Pi_{S}(T)$, is defined as
$$
\Pi_{S}(T)=\{s\in \mathbb{H}\ :\  T^2 - 2\Re(s)T + |s|^2\mathcal{I} \text{ is not bounded from below}\}.
$$
The compression $S$-spectrum of $T$, denoted by $\Gamma_{S}(T)$, is defined as
$$
\Gamma_{S}(T)=\{s\in \mathbb{H}\ :\   \text{ the range of }\ T^2 - 2\Re(s)T + |s|^2\mathcal{I} \text{ is not dense}\}.
$$
 \end{definition}

 From the definition it follows that
 $$
 \Pi_{0,S}(T)\subset \Pi_{S}(T).
 $$
There are several basic statements about the S-spectrum which are the analogues of the corresponding well-known facts in the classic case. While the main ideas to prove the results below follow those of the complex case, the proofs contain some suitable substantial changes as they involve the pseudo-resolvent.
\\
We start by proving a result generalizing the Fredholm alternative theorem:
\begin{theorem}\label{Fredholm}
Let $T$ be a compact operator acting on a quaternionic Hilbert space $\mathcal H$ and $s\not=0$. If ${\rm ker}( T^2 -2{\rm Re}(s) T +|s|^2 \mathcal I)=\{0\}$ then $T^2 -2{\rm Re}(s) T +|s|^2 \mathcal I$ is invertible.
\end{theorem}
\begin{proof}
We divide the proof in 3 steps.
\\
Step 1. We set $A_s(T)= T^2 -2{\rm Re}(s) T +|s|^2 \mathcal I$, and we prove that if ${\rm ran}(A_s(T))=\mathcal H$ then ${\rm ker}( T^2 -2{\rm Re}(s) T +|s|^2 \mathcal I)=\{0\}$.\\
First of all we note that when $T$ is compact, also $T^2-2 s_0T$ is compact.\\
Then we define $\mathcal{Q}_{n,s}:={\rm ker}( T^2 -2{\rm Re}(s) T +|s|^2 \mathcal I)^n$, $n=1,2,\ldots$ and,
  by absurd, we assume that $\mathcal{Q}_{1,s}\not=\{0\}$ and $0\not=v_1\in\mathcal{Q}_{1,s}$: since ${\rm ran}(A_s(T))=\mathcal H$, we can find $v_2$ such that $A_s(T) v_2=v_1$ and that $A_s(T)^2 v_2=A_s(T)v_1=0$, i.e.  $v_2\in \mathcal{Q}_{2,s}$. Iterating the procedure, we can find $v_{n+1}\in\mathcal{Q}_{n+1,s}$ such that $A_s(T) v_{n+1}=v_n$, $n=1,2,\ldots$. In conclusion, $v_n\in \mathcal{Q}_{n,s}$ for all $n=1,2,\ldots$ and the smallest power that annihilates $v_n$ is the $n$-th power. Thus $\mathcal{Q}_{n,s}\subset \mathcal{Q}_{n+1,s}$ and the sequence $\{\mathcal{Q}_{n,s}\}$ is strictly increasing. We can form a sequence $\epsilon_1, \epsilon_2, \ldots $ such that $\epsilon_n\in \mathcal{Q}_{n,s}(T)$ for all $n$, and the elements of the sequence are orthonormal. Since $A_s(T) \epsilon_{n+1}\in \mathcal Q_{n,s}$, $A_s(T)\epsilon_{n+1}$ is orthogonal to $\epsilon_{n+1}$, so that
$$
\| (T^2 -2 s_0 T) \epsilon_{n+1}\|^2=
\| A_s(T) \epsilon_{n+1}- |s|^2 \epsilon_{n+1}\|^2=\|A_s(T) \epsilon_{n+1}\|^2+|s|^2 \|\epsilon_{n+1}\|^2 \geq |s|^2\not=0.
$$
Since $\epsilon_n\to 0$ weakly, this contradicts the fact that $T^2-2 s_0T$ is compact.
\\
Step 2.
We prove that $A_s(T)$ is bounded from below on ${\rm ker}(A_s(T))^\perp$.\\
By absurd, we assume that the assertion does not hold, and so there exist unit vectors $\epsilon_n\in {\rm ker}(A_s(T))^\perp$ such that $A_s(T) \epsilon_n\to 0$.
\\
Since $T^2-2s_0T$ is compact we can assume, with no loss of generality, that the sequence $(T^2-2s_0T)\epsilon_n$ is (strongly) convergent to an element $\epsilon$. Thus we have
$$
|s|^2\epsilon_n=A_s(T)\epsilon_n -(T^2-2s_0T)\epsilon_n\to -\epsilon
$$
so that $\epsilon\in {\rm ker}(A_s(T))^\perp$ and it is a unit vector. However, $A_s(T) \epsilon_n=A_s(T) \epsilon$ so that $A_s(T) \epsilon=0$. This fact yields $\epsilon\in {\rm ker}(A_s(T))$ hence $\epsilon=0$. This contradicts the fact that $\epsilon$ has norm $1$ and the assertion follows.
\\
Step 3. We show that Step 1 and Step 2 apply when we consider $T^*$, $A_s(T)^*$ instead of $T$, $A_s(T)$.\\
We note that if Step 1 and 2 hold, then since ${\rm ran}(A_s(T))=A_s(T)\Big({\rm ker}(A_s(T))^\perp\Big)=\mathcal H$ and $A_s(T)$ is bounded from below then ${\rm ran}(A_s(T))$ is closed and so also ${\rm ran}(A_s(T)^*$ is closed.
Assume that ${\rm ker}(A_s(T))=\{0\}$. Then ${\rm ran}(A_s(T))^*$ is dense in $\mathcal H$ and since it is closed, it follows that ${\rm ran}(A_s(T)^*)=\mathcal H$ and so Step 1 applies to $(A_s(T)^*$. We then conclude that ${\rm ker}(A_s(T)^*)=\{0\}$ and also that ${\rm ker}(A_s(T)^*)^\perp=\mathcal H$. Then Step 2 can be applied to $A_s(T)^*$ leading to the conclusion that $A_s(T)^*$ is bounded from below. Theorem \ref{Th:3.1} yields that $A_s(T)^*$ is invertible and so also $A_s(T)$ is invertible and this concludes the proof.
\end{proof}
\begin{theorem}\label{Th:3.3}
Let $T\in\mathcal{B}(\mathcal H)$. Then $\sigma_S(T)=\Pi_{S}(T)\bigcup \Gamma_{S}(T)$.
\end{theorem}
\begin{proof}
We have to show that if $s\not\in \Pi_{S}(T)$ and $s\not\in\Gamma_{S}(T)$ then $s\not\in\sigma_S(T)$. But $s\not\in\Pi_{S}(T)$ implies that the range of
 $T^2 - 2\Re(s)T + |s|^2\mathcal{I}$ is closed. Since $s\not\in\Gamma_{S}(T)$ the range of $T^2 - 2\Re(s)T + |s|^2\mathcal{I}$ is dense, so
 we have that $T^2 - 2\Re(s)T + |s|^2\mathcal{I}$ is one-to-one and onto, thus it is invertible.
\end{proof}
\begin{theorem}[Weyl's theorem] \label{Th:Weyl} If $A \in \cB(\cH)$ and $K$ is a compact operator, then $$\sigma_S(A+K) \subset \sigma_S(A) \cup \Pi_{0,S}(A+K).$$
\end{theorem}
\begin{proof} Assume $s \in \sigma_S(A+K) \setminus \sigma_S(A).$  Then, since $A^2-2 s_0 A+|s|^2\mathcal{I}$ is invertible by assumption and $AK+KA$ is compact since $K$ is compact, we have
\[
\begin{split}
(A+K)^2-2& s_0 (A+K)+|s|^2\mathcal{I}= A^2-2 s_0 A+|s|^2\mathcal{I}+(K^2+AK+KA -2s_0K)\\
&=(A^2-2 s_0 A+|s|^2\mathcal{I})(\mathcal{I}+(A^2-2 s_0 A+|s|^2\mathcal{I})^{-1}(K^2+AK+KA -2s_0K)).
\end{split}
\]
 We conclude, by the invertibility of $A^2-2 s_0 A+|s|^2\mathcal{I}$, that
 $\mathcal{I}+(A^2-2 s_0 A+|s|^2\mathcal{I})^{-1}(K^2+AK+KA -2s_0K)$
 cannot be invertible. It follows from Theorem \ref{Fredholm} that $-1$ is an eigenvalue of the compact operator $(A^2-2 s_0 A+|s|^2\mathcal{I})^{-1}(K^2+AK+KA -2s_0K)$.
 Hence there exists a non-zero $v$ such that $(A^2-2 s_0 A+|s|^2\mathcal{I})^{-1}(K^2+AK+KA -2s_0K)v=-v$ which implies
 $$(K^2+AK+KA -2s_0K)v+(A^2-2 s_0 A+|s|^2\mathcal{I})v=0$$
 i.e.
 $$
 ((A+K)^2-2 s_0 (A+K)+|s|^2\mathcal{I})v=0.
 $$
Thus $ ((A+K)^2-2 s_0 (A+K)+|s|^2\mathcal{I})$ has a non-trivial nullspace and $s \in \Pi_{0,S}(A+K).$
\end{proof}
We also need a result on the boundary of the S-spectrum which will be important for the study of invariant subspaces.
Let $U$ be a subset in $\mathbb{H}$. We define the point-set boundary of $U$ as
$\partial U=\overline{U}\bigcap (\overline{\mathbb{H}\setminus U})$.
\begin{theorem}\label{Th:3.5}
Let $T\in \mathcal{B}(\mathcal H)$. Then we have $\partial \sigma_S(T)\subset \Pi_{S}(T)$.
\end{theorem}
\begin{proof}
First observe that $\partial \sigma_S(T)= \sigma_S(T)\bigcap \overline{\rho_S(T)}$.
Now assume that $s\in \partial \sigma_S(T)$ and $s\not\in \Pi_{S}(T)$.
We can take a sequence $s_n$ in the set $\rho_S(T)$ such that $s_n\to s$.

Step 1. We show that there exists $K>0$ and a positive integer $N$ such that $n\geq N$ implies
$$
\|(T^2 - 2\Re(s_n)T + |s_n|^2\mathcal{I})v\|\geq K\|v\|, \mbox{ for all }v\in \mathcal H.
$$
Suppose this does not hold. Then for all positive integers $m$ and $N$ there would exists an $n\geq N$
  and  a vector $v_m$ with $\|v_m\|=1$ such that
$$
\|(T^2 - 2\Re(s_n)T + |s_n|^2\mathcal{I})v_m\|\leq 1/m.
$$
But observe that
\[
\begin{split}
&(T^2 - 2\Re(s)T + |s|^2\mathcal{I})v_m
\\
&
=(T^2 - 2\Re(s_n)T + |s_n|^2\mathcal{I}  )v_m  +2(\Re(s_n)-\Re(s))Tv_m +(|s|^2- |s_n|^2)v_m.
\end{split}
\]
Taking the norm we have
\[
\begin{split}
\|(T^2& - 2\Re(s)T + |s|^2\mathcal{I})v_m\|
\\
&
\leq \|(T^2 - 2\Re(s_n)T + |s_n|^2\mathcal{I}  )v_m \| +2|\Re(s_n)-\Re(s)|\|T\|\|v_m\| +||s|^2- |s_n|^2|\|v_m\|
\\
&
\leq
\frac{1}{m}+2|\Re(s_n)-\Re(s)|\|T\| +||s|^2- |s_n|^2|
\end{split}
\]
and this implies that $s\in \Pi_{S}(T)$. So such $K$ and $N$ do not exist.

Step 2. We can show now that $s\not\in \Gamma_S(T)$. For if $v\in \mathcal H$, then for each $n$ there is a $w_n$ with
$$
(T^2 - 2\Re(s_n)T + |s_n|^2\mathcal{I} )w_n=v
$$
but since
$$
\|(T^2 - 2\Re(s_n)T + |s_n|^2\mathcal{I}  )w_n\|\geq K\|w_n\|
$$
we have
$$
\|w_n\|\leq \frac{1}{K}\|v\|, \ \ for \ \ n\geq N.
$$
Now observe that
\[
\begin{split}
\|(T^2& - 2\Re(s)T + |s|^2\mathcal{I} )w_n-v\|
\\
&
=\|(T^2 - 2\Re(s_n)T + |s_n|^2\mathcal{I}  - 2\Re(s)T + |s|^2\mathcal{I}  -(- 2\Re(s_n)T + |s_n|^2\mathcal{I} ))w_n-v\|
\\
&
\leq
\|(T^2 - 2\Re(s_n)T + |s_n|^2\mathcal{I})w_n -v\|
+2|\Re(s_n)- \Re(s)| \|T\|\|w_n\|  + ||s|^2 - |s_n|^2 |\|w_n\|
\\
&
\leq 0+2|\Re(s_n)- \Re(s)| \|T\|\|w_n\|  + ||s|^2 - |s_n|^2 |\|w_n\|
\\
&
\leq (2|\Re(s_n)- \Re(s)| \|T\|  + ||s|^2 - |s_n|^2 |)\frac{1}{K}\|v\|.
\end{split}
\]
If $n$ is large the term $(2|\Re(s_n)- \Re(s)| \|T\|  + ||s|^2 - |s_n|^2 |)\frac{1}{K}\|v\|$ is arbitrary small and it follows that
$s\not\in  \Gamma_{S}(T)$. Hence $s\in \partial \sigma_{S}(T)$ and $s\not\in  \Gamma_{S}(T)$ this implies that $s\not\in \sigma_S(T)$. But this contradicts Theorem \ref{Th:3.3}.
\end{proof}

The above theorem gives information about the spectra of the restrictions to invariant subspaces.
\begin{definition}\label{Def:3.6}
Let $T\in \mathcal{B}(\mathcal H)$ the full $S$-spectrum of $T$, denoted by $\eta(\sigma_S(T))$, is the union of $\sigma_S(T)$ and all bounded components of $\rho_{S}(T)$.
\end{definition}

This means that $\eta(\sigma_S(T))$ is the $S$-spectrum together with the holes in $\sigma_S(T)$.
\\
We now recall some basic facts useful to define the $S$-functional calculus.
\begin{definition}[$T$-admissible slice domain] \label{Def:3.10}
Let $T\in\boundOP(\mathcal H)$. A bounded axially symmetric domain $U\subset\mathbb{H}$ is called {\em $T$-admissible}
if $\sigma_S(T)\subset U$ and $\partial(U\cap\mathbb{C}_I)$ is the union of a finite number of piecewise continuously differentiable Jordan curves for any $I\in\mathbb{S}$.
\end{definition}

\begin{definition} \label{Def:3.11}
Let $T\in \mathcal{B}(\mathcal H)$.
\begin{enumerate}
\item[(i)] A function $f$ is called {\em locally left slice hyperholomorphic} on $\sigma_S(T)$, if there exists a $T$-admissible slice domain $U\subset\mathbb{H}$ such that $f\in \mathcal{SH}_L(\overline{U})$. We denote the set of all locally left slice hyperholomorphic functions on $\sigma_S(T)$ by $\mathcal{SH}_L(\sigma_S(T))$.
\item[(ii)] A function $f$ is called {\em locally right slice hyperholomorphic} on $\sigma_S(T)$, if there exists a  $T$-admissible slice domain $U\subset\mathbb{H}$ such that $f\in\mathcal{SH}_R(\overline{U})$. We denote the set of all locally left slice hyperholomorphic functions on $\sigma_S(T)$ by $\mathcal{SH}_R(\sigma_S(T))$.
\item[(iii)]  By $\mathcal{N}(\sigma_S(T))$ we denote the set of all functions $f\in\mathcal{SH}_L(\sigma_S(T))$ such that there exists a  $T$-admissible slice domain $U$ with $f(U\cap\mathbb{C}_\BI)\subset\mathbb{C}_\BI$ for all $I\in\mathbb{S}$.
\end{enumerate}
\end{definition}

\begin{definition}[S-functional calculus]\label{SCalc}
Let $T\in\mathcal{B}(\mathcal H)$. For any $f\in\mathcal{SH}_L(\sigma_S(T))$, we define
\begin{equation}\label{SCalcL}
f(T) = \frac{1}{2\pi}\int_{\partial(U\cap\mathbb{C}_\BI)}S_L^{-1}(s,T)\,ds_\BI\,f(s),
\end{equation}
where $\BI$ is an arbitrary imaginary unit and  $U$ is an arbitrary {$T$-admissible slice domain} such that $f$ is left slice hyperholomorphic on $\overline{U}$.
For any $f\in\mathcal{SH}_R(\sigma_S(T))$, we define
\begin{equation}\label{SCalcR}
f(T) = \frac{1}{2\pi}\int_{\partial(U\cap\mathbb{C}_\BI)}f(s)\,ds_\BI\,S_R^{-1}(s,T),
\end{equation}
where $\BI$ is an arbitrary imaginary unit and  $U$ is an arbitrary $T$-admissible slice domain such that $f$ is right slice hyperholomorphic on $\overline{U}$.
\end{definition}
For more details, and for the Banach space setting, see~\cite{acgs},~\cite{acsbook}, and~\cite{CSS}.

\section{Quaternionic operators on a Hilbert space}

Let $\cH$ be a right Hilbert (separable) space over $\BH.$  Moreover, we assume that $\cB(\cH)$ has identity $\mathcal{I}.$ For the time being, we often omit the identity operator from the formulations, its presence being clear from context. Let us first define the notion of an invariant subspace.

 \begin{definition}\label{Def:4.1}
The subspace $\cM \subset \cH$ is invariant under the operator $T\in \cB(\cH) $ if $T x \in \cM$ for every $x \in \cM.$ The collection of all subspaces of $\cH$ invariant under $T$ is denoted as $Lat(T);$ if $\cB' \subset \cB(\cH)$ then $Lat(\cB') := \cap_{T \in \cB'} Lat(T).$
 \end{definition}

Here, we have immediately the following property.

\begin{theorem} \label{Th:4.2}
 Let $M\in Lat(T)$. Then $\sigma_S(T|_{M})\subset\eta(\sigma_S(T))$, where $\eta(\sigma_S(T))$ is the full $S$-spectrum of $T$.
\end{theorem}

\begin{proof} Observe that $s\in \Pi_S(T|_{M})$ implies that there exists a sequence $v_n$, with $\|v_n\|=1$, such that
$(T^2 - 2\Re(s)T + |s|^2\mathcal{I})v_n $ converges to zero, so $\Pi_S(T|_{M}) \subset\Pi_S(T)$. By Theorem \ref{Th:3.5} we also have that
$$
\partial\sigma_S(T|_{M}) \subset\Pi_S(T|_{M}) \subset\sigma_S(T).
$$
If $\sigma_S(T|_{M}) $ contained points of the unbounded component of $\rho_S(T)$, then $
\partial\sigma_S(T|_{M}) $ would have to meet the unbounded component of $\rho_S(T)$ too, and the above shows that this is impossible.
\end{proof}

Additionally, we have the following notion of hyperinvariant subspaces.

  \begin{definition}\label{Def:1.02a}
The subspace $\cM$ is hyperinvariant under the operator $T$ if $\cM \in Lat(B)$ for every $B$ which commutes with $T.$
 \end{definition}

Hyperinvariant subspaces of $T$ give information about the commutants of $T,$ that is, the set of all operators $B$ such that $[T,B]=TB-BT=0.$

For the $S$-spectrum we also need a statement about the spectral radius, see \cite{CSS}.

\begin{definition} \label{Def:4.5}
Let $T\in\mathcal{B}(\mathcal H)$. Then the S-spectral radius of $T$ is defined to be the nonnegative real number
$$r_S(T) = \sup\{|s|:s\in\sigma_S(T)\}.$$
\end{definition}

\begin{theorem} \label{Th:4.6}
For $T\in\mathcal{B}(\mathcal H)$, we have
$$ r_S(T) = \lim_{n\to\infty}\|T^n\|^{\frac{1}{n}}.$$
\end{theorem}
Observe that, for normal operators on a Hilbert space the above theorem means that the spectral radius is given by
$$
 r_S(T) = \|T\|,
$$
since $\|T^n\|=\|T\|^n$, $n\in \mathbb{N}$ as in the complex case.

In the following we are in need of a spectral mapping theorem for the pseudo-resolvent operator $Q_S(T)=(T^2-2\Re(s)T+|s|^2\mathcal{I})^{-1}$.
Unfortunately, as we have already observed in the introduction the function
$s\mapsto (p^2-2\Re(s)p+|s|^2)^{-1}$ defined for $p\not\in [s]$ is neither left nor right slice hyperholomorphic
so we cannot use the $S$-functional calculus to prove the spectral mapping theorem in the case of a quaternionic Banach space.
But since we are working in a Hilbert space and the operator $T$ is normal we can use the continuous functional calculus for normal operators to deduce the spectral mapping theorem that we need.
\\
In the sequel, we also need to define a left multiplication in a right quaternionic Hilbert space. We assume that $\mathcal H$ is separable. This is always possible once that a Hilbert basis $(e_n)_{n\in\mathbb N}$ has been fixed, see \cite{GMP,Viswanath} and it is defined by
$sv=\sum_{n\in\mathbb N} e_n s \langle e_n,v\rangle$
where $v=\sum_{n\in\mathbb N} e_n \langle e_n,v\rangle$.
\begin{theorem} \label{thm:Sept1ub1}
Let $T \in \mathcal{B}(\mathcal{H})$ be normal, and $\mathcal{H}$ is a quaternionic Hilbert space.
Then there exist uniquely determined operators $A$ and $B$ which both belong to $\mathcal{B}(\mathcal{H})$ and an operator $J \in \mathcal{B}(\mathcal{H})$ which is uniquely determined on $\{ {\rm Ker}(T - T^*) \}^{\perp}$ so that the following properties hold{\rm :}
$$
 T = A + J B.
$$
where $A$ is self-adjoint, $B$ is positive, $J$ is anti self-adjoint and unitary, and $A$, $B$ and $J$ mutually commute.
Moreover, for any fixed $\BJ \in \mathbb{S}$, there exists an orthonormal basis $\mathcal{N}_\BJ$ of $\mathcal{H}$ with the property that $J = L_\BJ$, where $ L_\BJ$ is the left multiplication operator by $\BJ\in \mathbb{S}$.
\end{theorem}
The decomposition $T = A + J B$ was first established by Teichm\"ueller \cite{Teichmueller} when $J$ is a partial isometry.
The continuous functional calculus for normal operators on a quaternionic Hilbert space
 is defined for the following class of continuous  quaternionic-valued functions.
\begin{definition} \label{def:Sept15kj2}
Let $\Omega \subseteq \mathbb{H}$ be an axially symmetric set and let $D\subseteq \mathbb{R}^2$ be such that
$$
D = \{ (u,v)\in \mathbb{R}^2: \text{$u+\BJ v \in \Omega$ for some $\BJ \in \mathbb{S}$}\}.$$ Let $\mathcal{S}(\Omega, \mathbb{H})$ denote the quaternionic linear space of slice continuous functions,
i.e., $\mathcal{S}(\Omega, \mathbb{H})$ consists of functions $f: \Omega \to \mathbb{H}$ of the form
$$f(u+\BJ v ) = f_0(u,v) + \BJ f_1(u,v) \quad {\rm for} \ \  (u,v)\in D \quad   {\rm and \ for}  \ \ \BJ \in \mathbb{S},$$
where $f_0$ and $f_1$ are continuous $\mathbb{H}$-valued functions on $D$ so that
$$f_0(u,v) = f_0(u,-v) \quad {\rm and} \quad f_1(u,v) = -f_1(u,-v).$$
If $f_0$ and $f_1$ are real-valued, then we say that the continuous slice function $f$ is {\it intrinsic}.
The subspace of intrinsic continuous slice functions is denoted by $\mathcal{S}_{\mathbb{R}}(\Omega, \mathbb{H})$.
\end{definition}

The following functional calculus will be useful for proving a spectral theorem for a normal operator $T \in \mathcal{B}(\mathcal{H})$.

\begin{theorem}[Theorem 7.4 in \cite{GMP}]\label{Continuous} Let $T \in \mathcal{B}(\mathcal{H})$ be a normal operator. There exists a unique continuous *-homomorphism
$$
\Psi_{\mathbb{R}, T}: f \in \mathcal{S}_{\mathbb{R}} (\sigma_S(T), \mathbb{H}) \mapsto f(T) \in \mathcal{B}(\mathcal{H})
$$
of real-Banach unital $C^*$-algebras such that (Spectral Mapping Theorem):
\begin{equation}\label{spectmapN}
\sigma_S(f(T)) = f(\sigma_S(T)).
\end{equation}
\end{theorem}
We are now in the position to prove an important estimate on the pseudo-resolvent operator
that will be used in the sequel to prove on of the main results of this paper.

\begin{lemma} \label{Lm:4.12}
Let $T$ be a bounded normal linear operator on a quaternionic Hilbert space $\mathcal H$ and $s\in \mathbb{H}$.
Then we have
$$
\|(T^2-2\Re(s)T+|s|^2)^{-1}\|\leq \frac{1}{{\rm dist}(\sigma_S(T),[s])^2},
$$
where we have set
$$
{\rm dist}(\sigma_S(T),[s]):=\inf \{|w-p|,\ w\in\sigma_S(T),\, p\in [s]\}.
$$
\end{lemma}
\begin{proof}
Since $T$ is a normal operator also $(T^2-2\Re(s)T+|s|^2)^{-1}$ is a normal operator  so the $S$-spectral radius gives
$$
\|(T^2-2\Re(s)T+|s|^2)^{-1}\|=\sup\{|w| \ :\ w\in \sigma_S((T^2-2\Re(s)T+|s|^2)^{-1}) \}.
$$
From the Spectral Mapping Theorem (Theorem~\ref{Continuous}) we obtain
$$
\sup\{|w| \ :\ w\in \sigma_S((T^2-2\Re(s)T+|s|^2)^{-1}) \}=\frac{1}{\inf\{ |w| \ :\ w\in \sigma_S(T^2-2\Re(s)T+|s|^2)\}}
$$
and using again Theorem~\ref{Continuous} we get
$$
\frac{1}{\{\inf |w| \ :\ w\in \sigma_S(T^2-2\Re(s)T+|s|^2)}=\frac{1}{\inf\{ |w^2-2\Re(s)w+|s|^2| \ :\ w\in \sigma_S(T)\}}.
$$
From the above equalities we obtain
$$
\|(T^2-2\Re(s)T+|s|^2)^{-1}\|=\frac{1}{\inf\{ |w^2-2\Re(s)w+|s|^2| \ :\ w\in \sigma_S(T)\}}.
$$
Furthermore, we have
\[
\begin{split}
\inf\{ &|w^2-2\Re(s)w+|s|^2| \ :\ w\in \sigma_S(T)\}=\inf\{ |(w-p)*(w-\bar p)| \ :\ w\in \sigma_S(T), \, p\in [s]\}\\
&= \inf\{ |(w-p)(\tilde w-\bar p)| \ :\ w\in \sigma_S(T), \tilde w\in [w], \, p\in [s]\}\\
&\geq \inf\{ |w-p| \ :\ w\in \sigma_S(T), \tilde w\in [w], \, p\in [s]\} \inf\{|\tilde w-\bar p| \ :\ w\in \sigma_S(T), \tilde w\in [w], \, p\in [s]\}\\
&= {\rm dist}(\sigma_S(T), [s])^2
\end{split}
\]
which leads to the statement.
\end{proof}

\section{Some results on Schatten class of quaternionic operators}

Since we shall discuss compact perturbations of normal operators let us recall some basic statements on Schatten classes of quaternionic operators. These classes have been recently introduced in the paper \cite{CGJ}.

We denote  by $\mathcal{B}_0(\mathcal{H})$ the set of all compact quaternionic right linear operators on $\mathcal{H}$. For an anti-selfadjoint unitary operator $J$, we define the set
\[\mathcal{B}_J(\mathcal{H}) := \{T\in\mathcal{B}(\mathcal{H}): [T,J] = 0\}.\]

Consider now an arbitrary compact operator $T$. We can find a Hilbert-basis $(e_n)_{n\in\mathbb{N}}$ and an orthonormal set $(\sigma_n)_{n\in\mathbb{N}}$ in $\mathcal{H}$ such that
\begin{equation}\label{SVD}
Tx = \sum _{n\in\mathbb{N}} \sigma_n\lambda_n \langle e_n, x\rangle\qquad \forall x \in \mathcal{H},
\end{equation}
where the $\lambda_n\in\mathbb{R}^+$ are the singular values of $T$, i.e. the eigenvalues of the operator $|T|:=\sqrt{T^*T}$ in non-increasing order,
 the vectors $(e_n)_{n\in\mathbb{N}}$ form an eigenbasis of $|T|$ and $\sigma_n = W e_n$ with $W$ unitary on $\ker W ^{\perp}$ and such that $T = W|T|$. See \cite{DS1963} and Remark 3.4 in \cite{CGJ}.
\begin{definition} \label{Def:5.1}
Let $J \in \mathcal{B}(\mathcal{H})$ be an anti-selfadjoint and unitary operator.
For $p\in(0,+\infty]$, we define the $(J,p)$-Schatten class of operators $S_p(J)$ as
$$
S_p(J) := \{ T \in \mathcal{B}_0(\mathcal{H}) :  [T,J]=0 \ \text{and}\ (\lambda_n(T))_{n\in\mathbb{N}}\in\ell^p\},
$$
where $(\lambda_n(T))_{n\in\mathbb{N}}$ denotes the sequence of singular values of $T$
and $\ell^p$ and $\ell^\infty$ denote the space of $p$-summable resp. bounded sequences.
For $T\in S_p(J)$, we introduce the following norms
\begin{equation}\label{pNorm}
\|T\|_p = \left(\sum _{n \in \mathbb{N}} |\lambda_n(T)|^p \right)^{\frac{1}{p}}\qquad\text{if }p \in [1,+\infty)
\end{equation}
and
\begin{equation*}
\|T\|_{p} = \sup_{n\in\mathbb{N}}\lambda_n(T) = \|T\|\qquad\text{if } p = +\infty.
\end{equation*}
\end{definition}
Using this norms we can give the following statements whose proofs are straightforward modifications of the classic proofs.
\begin{lemma} \label{Lm:5.2}
Let $T\in S_p(J)$, then
there exists a sequence of finite rank operators $\{T_n\}_{n\in \mathbb{N}}$ such that
$$
\|T-T_n\|\to 0\ \ \ and \ \ \ \|T-T_n\|_p\to 0, \ \ as\ \ \ n\to \infty .
$$
\end{lemma}
\begin{proof}
It follows as in the classical case, see Lemma 11 p. 1095 in \cite{DS1963}.
\end{proof}
\begin{lemma} \label{Lm:5.3}
Let $T\in S_p(J)$, then for every operator $A\in \mathcal{B}(\mathcal{H})$ the operators $AT$ and $TA$ belong to  $S_p(J)$ and
$$
\|AT\|_p\leq \|A\|\|T\|_p \ \  and \ \ \ \|TA\|_p\leq \|A\|\|T\|_p
$$
\end{lemma}
\begin{proof}
From Lemma  3.7 in \cite{CGJ} it has been established that if $T$ is a positive compact operator the singular values $\lambda_{n+1}$ are given by
$$
\lambda_{n+1}=\min_{y_1,...,y_n}\ \ \max_{\langle x_i,y_i\rangle, i=1,...,n }\frac{\|Tx\|}{\|x\|}
$$
so using this formula we can define singular values also for operators which are not compact. Furthermore, with a similar proof as in the complex case we can state that for compact or non-compact bounded operators we can extend Corollary 3.9  in \cite{CGJ}. More precisely we have:
\begin{equation}
\lambda_{n+m+1}(T_1+T_2) \leq \lambda_{n+1}(T_1) + \lambda_{m+1}(T_2)
\end{equation}
and
\begin{equation}
\lambda_{n+m+1}(T_1T_2) \leq \lambda_{n+1}(T_1)\lambda_{m+1}(T_2).
\end{equation}
As a particular case we have
$$
\lambda_n(TA)\leq \lambda_n (T)\|A\|_p,\ \ \ \lambda_n(AT)\leq \lambda (T)\|A\|_p\lambda (T),
$$
so we get the statement.
\end{proof}
For the following definition we have to stress the fact that the Schatten classes under consideration are restricted to $k \in \BN.$
\begin{definition} \label{Def:5.4}
Let $\BI\in\mathbb S$ be any arbitrary, but fixed element in $\mathbb S$.
Let $T\in S_k(J)$ with  $k\in \mathbb{N}$, and let $\{s_1,s_2, \ldots\}$ be an enumeration of the non-zero elements in $\Pi_{0,S}(T)\cap\mathbb{C}_{\BI}$ repeated according to their multiplicity. We define
$$
\delta_{k,\BI}(T)=\Pi_{l =1}^\infty \left[(1+s_l) \exp\left(-s_l+\frac{s_l^2}{2}+\cdots +(-1)^{k-1}\frac{s_l^{k-1}}{k-1}\right)\right].
$$
In case of $\Pi_{0,S}(T)\cap\mathbb{C}_{\BI}=\{0\}$ we define $\delta_{k,\BI}=1$.
\end{definition}

For this function we have the following result:
\begin{lemma}\label{unkown} \label{Lm:5.5}
Let $T\in S_k(J)$ with  $k\in \mathbb{N}$. Then:
\begin{enumerate}
\item[(i)] $\delta_{k,\BI}(T)$ is an absolutely convergent infinite product;
\item[(ii)] there exists a constant $\Gamma_k$ depending only on $k\in\mathbb N$ such that
$$
|\delta_{k,\BI}(T)|\leq \exp(\Gamma_k \|T\|_k^k);
$$
\item[(iii)] $\delta_{k,\BI}(T)$ is a continuous function in the topology of $S_k(J)$;
\item[(iv)] there exists a constant $M_k$ depending only on $k\in\mathbb N$ such that
$$
\|\delta_{k,\BI}(T)(\mathcal{I}+T)^{-1}\|\leq \exp(M_k \|T\|_k^k),
$$
when $-1\not\in \sigma_S(T)$.
\end{enumerate}
\end{lemma}
\begin{proof}
Since the function $\delta_{k,\BI}(T)$ has values in $\mathbb C_{\BI}$ the proof of $(i)$ is the same proof as given for Lemma 22, (a). p. 1106 in \cite{DS1963}.
Point (ii) is  Lemma 22, (b). p. 1106 in \cite{DS1963},  since we consider restrictions to the complex plane $\mathbb C_{\BI}$
point (iii) is Lemma 22, (c). p. 1106 in \cite{DS1963}, finally point (iv) is Theorem 24 p. 1112.
\end{proof}

Remark, moreover, that $\delta_{k,\BI}$ does not depends on the chosen element $\BI \in \BS$.

\section{Perturbation of quaternionic normal operators}

Take $\BI \in \BS$ and consider the slice $\BC_{\BI}.$ Let $\cC$ be an exposed arc in $\sigma_S(T)\cap \BC_{\BI}$ for all $\BI \in \BS,$ that is to say, there exists an open disk $\BD_{\BI}$ such that $\BD_{\BI} \cap \sigma_S(T)  = \cC $ and $\cC$ is a smooth Jordan arc in $\BC_{\BI}, \BI \in \BS.$ Furthermore, for a curve $C$ we will denote by $\tilde C$ its axially symmetric completion.
 \begin{definition}
 The distance between equivalence classes $[s]$, $[t]$ is defined as
 $$
 |[s]-[t]|=\inf_{s\in [s],t\in [t]}|s-t|.
 $$
\end{definition}
\begin{theorem}\label{Th:6.1} Let $T \in \cL(\cH)$ be such that $\sigma_S(T)$ contains the axially symmetric completion $\tilde\cC$ of an exposed arc $\cC,$ and let $k \in \BN.$ If for each $[s_0] \in \tilde\cC$ and each axially symmetric completion $\tilde L$ of a closed line segment $L$ not tangent to $\cC$ and satisfying to $\tilde L \cap \sigma_S(T) = \{[ s_0] \},$ there exists a constant $K>0$ such that
$$\| S^{-1}_L (s, T) \| \leq \exp(K |[s]-[s_0]|^{-k}),$$ for all $s \in \tilde L \setminus \{ [s_0] \},$ then $T$ has a non-trivial hyperinvariant subspace.
\end{theorem}
\begin{proof} Let $\BI \in \BS$ arbitrary, but fixed. We can assume that $\cC_\BI=\tilde\cC\cap\BC_{\BI}$ has a representation (as a smooth Jordan arc in a given slice) $s= q(t), t \in (0,1),$ with $q$ one-to-one, $|q'(t)| < \tan(\pi/ 5k), t \in (0,1)$ and where $q''(t)$ exists everywhere in $(0,1).$

If $\BD_\BI$ is an open disk in $\mathbb C_{\mathbb I}$ such that $\BD_\BI \cap \sigma_S(T)  = \cC$ and $\cC$ is a smooth Jordan arc  then $\BD_\BI $ is the union of disjoint Jordan regions $\BD_{1,\BI}, \BD_{2,\BI}$ lying above and below $\cC,$ respectively.

Consider subarcs $\cJ$ such that $\ov{\cJ} \subset \tilde\cC\cap \BC_{\BI},$ and with endpoints $s_1$ and $s_2$ (with Re$s_1 < $ Re$s_2$.) Construct a simple closed Jordan polygon $\Gamma_1(\cJ) \in \BD_{\BI}$ enclosing $\cJ$ and intersecting $\cC$ at $s_1$ and $s_2$ only. Assume in the construction of $\Gamma_1(\cJ) \in \BD_{\BI}$ that the angles at $s_1$ have arguments $\pm \pi/5k$ while the angles at $s_2$ have arguments $\pi \pm \pi/5k.$ Then these lines generates a hexagon lying in $\BD_{\BI}.$ Moreover, by interchanging the angles at $s_1$ and $s_2$ one obtains a second polygon $\Gamma'_1(\cJ) \in \BD_{\BI}.$ Let $\Gamma_2(\cJ) \in \BD_{\BI}$ be the union of $\Gamma'_1(\cJ)$ and any fixed circle containing $\BD_{\BI} \cap \sigma_S(T)$ in its interior.

Fix an open subarc $\cJ_0$ of $\cC$ such that $\ov{\cJ}_0 \subset \cC.$ Let  $S^{-1}_L (s, T) $ denote the slice-hyperholomorphic quaternionic-valued function taking the resolvent set $\rho_S(T)$ into $\cH.$

We recall that, due to the non-commutativity of the quaternions, when we consider the $S$-resolvent operator
$S^{-1}_L(s,T)$,  which is right slice hyperholomorphic, the function $S^{-1}_L(s,T)x$ cannot be right slice hyperholomorphic.
To avoid this problem we consider the subset  $\cH_\mathbb{R}$ of $\mathcal{H}$ defined as:
$$
\cH_\mathbb{R}:=\{ x\in \cH\ : \ xp=px,\ \ \forall p\in \mathbb{H}\ \},
$$
where we are using both the left and right multiplication in $\mathcal H$ and thus we have fixed a Hilbert basis.\\
Note that $\cH=\sum_{i=0}^3 \cH_\mathbb{R}e_i$ where we set $e_0=1$. In fact, given any $x\in\cH$ we can define the element ${\rm Re}(x)= \frac 14(\sum_{i=0}^3 \bar{e_i}xe_i)\in\cH_\mathbb{R}$ and we have that $x=\sum_{i=0}^3 {\rm Re}(\bar{e_i}x) e_i$.
If two right linear operators $T$ and $\tilde{T}$ agree on $\cH_\mathbb{R}$ then they agree on the entire $\cH$, see \cite{acsbook}, p. 176.

Let us define the sets
\begin{equation}
\sN_{R, \mathbb R} = \{ x \in \cH_\mathbb{R} : S^{-1}_R(s,T)x \mbox{ has a left slice-hyperholomorphic extension to  }(\ov{\cJ}_0)^c\}, \label{NR}
\end{equation}
\begin{equation}
 \sN_{L, \mathbb R}= \{ x \in \cH_\mathbb{R} : S^{-1}_L(s,T)x \mbox{ has a right slice-hyperholomorphic extension to  }(\ov{\cJ}_0)^c \}
\label{NL}
\end{equation}
where $(\ov{\cJ}_0)^c$ denotes the complement set of $\ov{\cJ}_0$.
Then, $\sN_{L, \mathbb R} = \sN_{R, \mathbb R}$. Indeed, we have for $x \in \sN_{L, \mathbb R}$ that there exists a slice hyperholomorphic  continuation $$f(s) = S^{-1}_L(s,T) x$$
 to $(\ov{\cJ}_0)^c$. In a similar way, for $x \in \sN_{R, \mathbb R}$ there exists a slice hyperholomorphic  continuation  to $(\ov{\cJ}_0)^c$ $$\tilde f(s) = S^{-1}_R(s,T) x.$$
Let us prove that $\sN_{L, \mathbb R} \subseteq \sN_{R, \mathbb R}:$ Consider $x \in \sN_{L,, \mathbb R}$.  Then $$x= (s-T)\star_r S^{-1}_L(s,T) x = (s-T)\star_r f(s)$$ so that $$S^{-1}_R(s,T) x = S^{-1}_R(s,T) (s-T)\star_r f(s) := \tilde f(s)$$ is a slice hyperholomorphic continuation and $x \in \sN_{R, \mathbb R}.$ In a similar way, $\sN_{R, \mathbb R} \subseteq \sN_{L, \mathbb R}$, note that these sets are real subspaces of $\cH$.
\\
In order to construct a subspace of $\cH$ as a quaternionic linear space, we recall that each $x\in\cH$ can be uniquely decomposed as $x=\sum_{i=0}^3 x_ie_i$, if we set
$$
\sN_{R}=\sum_{i=0}^3\sN_{R, \mathbb R}e_i, \qquad \sN_{L}=\sum_{i=0}^3\sN_{L, \mathbb R}e_i,
$$
and $e_0=1$, we deduce from the previous discussion that $\sN_R=\sN_L$.
From now on, we denote these sets as $\sN.$ \\
We point out that in the case of $\sN_R$ one could have introduced it directly:
\begin{equation}
\sN_{R} = \{ x \in \cH : S^{-1}_R(s,T)x \mbox{ has a left slice-hyperholomorphic extension to  }(\ov{\cJ}_0)^c\},
\end{equation}
since there are no issues of loosing the left hyperholomorphy of $S^{-1}_R(s,T)$ by letting it to act on $x\in\cH$, however to show the equality $\sN_R=\sN_L$ it is more convenient to proceed as above.
We now note that it is not immediate to prove the hyperinvariance of $\sN$, in fact, in general, $A$ does not commute with the operator $\bar s\mathcal I$. Thus we consider
\begin{equation}
\sM = \{ x \in \cH : Q_s(T)x \in \cH\ {\rm for\ all}\ s\in (\ov{\cJ}_0)^c \}.
\label{N}
\end{equation} Obviously, as $AT=TA,$ one gets
\begin{gather*}  (T^2-2s\Re(s)T + |s|^2)^{-1} (Ax) = A (T^2-2\Re(s) T + |s|^2)^{-1} x.
\end{gather*} Hence,  $\sM$ is invariant under every operator $A$ which commutes with $T.$ We now show that $\sN=\sM$. Every $x\in\sM$ is such that $Q_s(T)x\in\cH$ for all $s\in (\ov{\cJ}_0)^c $ thus
$$
-(T-\bar s)Q_s(T)x= S^{-1}_R(s,T)x\in\cH\ \ \ \text{for all}\ \ \  s\in (\ov{\cJ}_0)^c$$
 and so $x\in\sN$.

Conversely, consider $x\in\sN=\sN_R$, then $-(T-\bar s)Q_s(T)x$ admits left slice hyperholomorphic extension to $(\ov{\cJ}_0)^c$, so $-(T-\bar s)Q_s(T)x\in\cH$ and $Q_s(T)x$ belongs to the domain of $T$ which is $\cH$ and thus
$Q_s(T)x\in\cH$ for $s\in (\ov{\cJ}_0)^c$. Thus $\sN=\sM$ and the hyperinvariance of $\sN$ is proved.

We define the functions $m_r, ~m_l$ as
\begin{gather}
m(s):= \left\{    \begin{array}{cl}
                        \exp_{\star} \left[ -(s-s_1)^{\star (-2k)}  -(s-s_2)^{\star (-2k)}   \right], & s\not=s_1, s_2 \\
                        0 &   s=s_1, s_2
                        \end{array}      \right. \label{Eq:measure}
\end{gather} with
$$\exp_{\star} \left[ -(s-s_1)^{\star (-2k)}  -(s-s_2)^{\star (-2k)}   \right] := \sum_{n=0}^{\infty} \frac{1}{n!} \left[ -(s-s_1)^{\star (-2k)}  -(s-s_2)^{\star (-2k)}   \right]^{\star n}$$ where $m_r$ or $m_l$ correspond to the appropriate $\star_r$ or $\star_l$ multiplication.

Let $G$ be an open annulus-like region whose boundary is $\Gamma_2(\cJ)$. After a judicious choice for $\Gamma_1(\cJ)$ at $s_1, s_2$ it follows that both $m_r$ and $m_l$  are slice-hyperholomorphic on $\tilde G$ and continuous on $\ov{\tilde G}.$ Take now a sequence $(x_n)_{n\in \BN}$ in $\sN$ such that $x_n \rightarrow x \in \cH.$ Moreover, denote by $S^{-1}_{L, n}(s, T)(x_n)$ the slice-hyperholomorphic continuation of $S^{-1}_{L}(s, T)(x_n)$ to the complement of $\ov{\cJ}_0.$ Then
\begin{gather*}
\| S^{-1}_{L, n}(s, T) \star_r m_r(s) - S^{-1}_{L, k}(s, T) \star_r  m_r(s) \| \leq \sup_{w \in \Gamma_2(\cJ)} \left\| \left( S^{-1}_{L, n}(w, T)  - S^{-1}_{L, k}(w, T) \right) \star_r m_r(w) \right\|
\end{gather*}  by the maximum modulus principle. By the hypothesis on the growth of the pseudo-resolvent of $T$,  given a line $L \subset \Gamma_2(\cJ)$ with endpoint on $s_1$  we have, for $w \in L \setminus \{ s_1 \}$:
\[
\begin{split}
&\left\| \left( S^{-1}_{L, n}(w, T)  - S^{-1}_{L, k}(w, T) \right) \star_r m_r(w) \right\|
 \\
 &
 = \left\| \left( S^{-1}_{L}(w, T)  - S^{-1}_{L}(w, T) \right) \star_r m_r(w) (x_n-x_k)  \right\|
 \\
&\leq \| x_n-x_k \| \left\| \exp_{\star_r} \left[ -(w-s_1)^{\star_r (-2k)}  -(w-s_2)^{\star_r (-2k)}  \right]   \right\|   \exp( K |w-s_1|^{-k} )        \\
 &
 \leq N \| x_n-x_k \|  \exp_{\star_r}  \left[  \mbox{Sc} \left( -(w-s_1)^{\star_r (-2k)}  \right)    \right]  \exp( K |w-s_1|^{-k} ),
\end{split}
\]
where
$$
N: = \sup_{w \in L} \left| \exp_{\star_r}  \left[  -(w-s_2)^{\star_r (-2k)}  \right]    \right|.
$$
As $w, s_1 \in L$ then we have $w- s_1 = |w- s_1| e^{I \theta}$ in the slice $\BC_{\BI}$ for a certain angle $\theta$ (either $\pi/5k$ or $-\pi/5k$).
Then the term
$$ \exp_{\star_r}  \left[  \mbox{Sc} \left( -(w-s_1)^{\star_r (-2k)}  \right)    \right]   \exp( K |w-s_1|^{-k} )$$ is bounded on $L$ and it exists $M_1 > 0$ such that
\begin{gather*}
\| S^{-1}_{L, n}(s, T) \star_r m_r(s) - S^{-1}_{L, k}(s, T) \star_r  m_r(s) \| \leq M_1 \| x_n - x_k \|, \quad \mbox{for all }w \in L.
\end{gather*}
In exactly the same manner, it can be shown that it exists a constant $M_2 >0$ such that it holds
\begin{gather*}
\| S^{-1}_{L, n}(s, T) \star_r m_r(s) - S^{-1}_{L, k}(s, T) \star_r  m_r(s) \| \leq M_2 \| x_n - x_k \|,
\end{gather*} for all $w$ on the lines of $\Gamma_2(\cJ)$ passing trough $s_2.$

Since a similar statement holds for $w$ on the remaining lines of $\Gamma_2(\cJ)$ we obtain that $\left( S^{-1}_{L, n}(s, T) \star_r m_r(s) \right)_{n \in \BN}$ is a uniform Cauchy sequence on $\ov {\tilde G},$ and it converges uniformly to a function $g(s)$ right slice hyperholomorphic on $\tilde G$ and continuous on $\ov {\tilde G}.$
Let us consider $x \in \cH_{\mathbb{R}}$ and set
\begin{gather*}
y= \int_{\partial( \Gamma(\cJ_0) \cap \BC_{\BI})}  S_L^{-1}(z,T) \star_r m(z) dz_{\BI} x.
\end{gather*}

Take an arbitrary $w \in \rho_S(T)$ outside $\Gamma_1(\cJ_0).$ We get, with all the $\star_r$-multiplications computed in the variable $z$:
\[
\begin{split}
S_R^{-1}(w,T)  y&= \int_{\partial( \Gamma_1(\cJ_0) \cap \BC_{\BI})}  S_R^{-1}(w,T) S_L^{-1}(z,T) \star_r m_r(z) dz_{\BI} x
\\
&= \int_{\partial( \Gamma_1(\cJ_0) \cap \BC_{\BI})} (z-w)^{-\star_r} \star_r [ S_L^{-1}(z,T)- S_R^{-1}(w,T)] \star_r m_r(z) dz_{\BI} x
\\
&= \int_{\partial( \Gamma_1(\cJ_0) \cap \BC_{\BI})} (z-w)^{-\star_r} \star_r S_L^{-1}(z,T) \star_r m_r(z) dz_{\BI} x
\\
&
\ \ \ \ \ \ \ \ \ \ -\underbrace{\int_{\partial( \Gamma_1(\cJ_0) \cap \BC_{\BI})} (z-w)^{-\star_r} \star_r  S_R^{-1}(w,T)  \star_r m_r(z) dz_{\BI} x}_{=0}
\\
&
= \int_{\partial( \Gamma_1(\cJ_0) \cap \BC_{\BI})} (z-w)^{-\star_r} \star_r S_L^{-1}(z,T) \star_r m_r(z) dz_{\BI} x.
\end{split}
\]
Now, the integrand has a slice hyperholomorphic continuation to the exterior of the curve $\Gamma_1(\cJ_0) $ so that it has a slice hyperholomorphic continuation to the complement of $\ov{\cJ_0},$ and $y \in \sN_{R,\mathbb R}.$

Now, let $x\in\cH_{\mathbb R}$ so that, recalling that  $\sN_{R, \mathbb R} = \sN_{L, \mathbb R},$ we have that
\begin{gather*}
y= \int_{\partial( \Gamma(\cJ_0) \cap \BC_{\BI})}  S_L^{-1}(s,T) \star_r m_r(s) x ds_{\BI}
\end{gather*} corresponds one-to-one to
\begin{gather*}
\tilde y= \int_{\partial( \Gamma(\cJ_0) \cap \BC_{\BI})} ds_{\BI} ~m_l (s) \star_r S_R^{-1}(s,T) x,
\end{gather*} that is to say, $$y = 0 \quad \Leftrightarrow \quad \tilde y =0.$$

Now, let $s_0 \in \cJ_0.$ Thus, $s_0 \in \Pi_S(T)$ and so, for all $\epsilon >0$ there exists a unit vector $x_\epsilon\in \mathcal{H}_\mathbb{R}$ such that
$$
(T^2-2\mbox{Re} (s_0) T + |s_0|^2\mathcal{I}) x_\epsilon = -(T-\ov{s}_0)h_\epsilon
$$
 with $\| h_\epsilon \| < \epsilon.$ Hence
\[
\begin{split}
\tilde y&= \int_{\partial( \Gamma(\cJ_0) \cap \BC_{\BI})} ds_{\BI} ~ m_l (s) \star_l S_R^{-1}(s,T) x_\epsilon
 \\
 &
 = \int_{\partial( \Gamma(\cJ_0) \cap \BC_{\BI})} ds_{\BI} ~  m_l (s) \star_l S_R^{-1}(s,T) S_L^{-1}(s_0, T) h_\epsilon
\\
&
= \int_{\partial( \Gamma(\cJ_0) \cap \BC_{\BI})} ds_{\BI} ~ m_l (s) \star_l \left[ S_L^{-1}(s_0, T) - S_R^{-1} (s,T)  \right] \star_l (s_0 -s)^{-\star_l}  h_\epsilon
\\
&
= \int_{\partial( \Gamma(\cJ_0) \cap \BC_{\BI})} ds_{\BI} ~ m_l (s) \star_l  S_L^{-1}(s_0, T) \star_l (s_0 -s)^{-\star_l}  h_\epsilon
\\
&
- \int_{\partial( \Gamma(\cJ_0) \cap \BC_{\BI})} ds_{\BI} ~ m_l (s) \star_l  S_R^{-1} (s,T)  \star_l (s_0 -s)^{-\star_l}  h_\epsilon.
\end{split}
\]
 Since
 $$
 S_L^{-1}(s_0, T) \star_l (s_0 -s)^{-\star_l}  = (s_0 -s)^{-\star_r} \star_r S_L^{-1}(s_0, T)
 $$
 where the $\star$ -multiplication  is taken with respect to $s$ in the right-hand side and  is taken with respect to $s_0$ in the left-hand side (see Remark \ref{star})
we have
\[
\begin{split}
\tilde y&= \int_{\partial( \Gamma(\cJ_0) \cap \BC_{\BI})} ds_{\BI} ~ m_l (s) \star_l [(s_0 -s)^{-\star_r} ] x_\epsilon - \int_{\partial( \Gamma(\cJ_0) \cap \BC_{\BI})} ds_{\BI} ~  m_l (s) \star_l  S_R^{-1} (s,T)  \star_l (s_0 -s)^{-\star_l}  h_\epsilon
 \\
&
= 2\pi i ~m_l(s_0) x_\epsilon - \int_{\partial( \Gamma(\cJ_0) \cap \BC_{\BI})} ds_{\BI} ~  m_l (s) \star_l  S_R^{-1} (s,T)  \star_l (s_0 -s)^{-\star_l}  h_\epsilon,
\end{split}
\]
by the Cauchy integral formula on slices. The function $m_l (s) \star_l  S_R^{-1} (s,T)$ is a continuous operator-valued function on $\Gamma(\cJ_0) \cap \BC_{\BI})$ while $(s_0 -s)^{-\star_l} $ is continuous on $\Gamma(\cJ_0) \cap \BC_{\BI}).$ Hence, we have for the last integral that
$$\left\|  \int_{\partial( \Gamma(\cJ_0) \cap \BC_{\BI})} ds_{\BI} ~  m_l (s) \star_l  S_R^{-1} (s,T)  \star_l (s_0 -s)^{-\star_l}  h_\epsilon   \right\| \rightarrow 0$$ as $\epsilon \rightarrow 0.$ Since $\| m_l(s_0) x_\epsilon \| = |m_l(s_0)| \not=0$ we obtain that the vector $$\tilde y= \int_{\partial( \Gamma(\cJ_0) \cap \BC_{\BI})} ds_{\BI} ~ m_l (s) \star_l S_R^{-1}(s,T) x_\epsilon$$ is non-zero for all $\epsilon >0,$ and hence, $\sN_{L, \mathbb R}$ and thus $\sN$ are non-trivial. \end{proof}

\begin{lemma} \label{Lem:6.2}
Let $T=A+B$, where $A$ is normal and $B$ is in the Schatten class for some integer $k>1$. Assume that $\Pi_{0,S}(T)=\emptyset$ and that $\sigma_S(A)$ contains the axially symmetric completion for an exposed arc $\cJ$. Let $s_0\in \tilde \cJ \cap \BC_{\BI}$ and $L$ be any closed bounded line segment starting from $s_0$ and not being tangent to $\tilde \cJ \cap \BC_{\BI}.$ Moreover, assume that $L\cap \sigma_S(A) =\{s_0\}$. Then there exists a constant $K$ such that for all $s\in L\setminus\{s_0\}$ we have
$$
\|Q_s(T)\|\leq \exp (K |[s]-[s_0]|^{-2k-2}).
$$
\end{lemma}

\begin{proof}
By Weyl's Theorem \ref{Th:Weyl} we know that $\Pi_{0,S}(T)=\emptyset$ means $\sigma_S(T)\subset \sigma_S(A)$. Consider the slice $\mathbb{C}_{\BI}$  where $\BI \in \BS$ is arbitrary but fixed. Take an open disk $\mathcal D\subset \mathbb{C}_{\BI}$ such that $\mathcal D\cap \sigma_S(A)=\tilde\cJ \cap \BC_{\BI}$. Here $\cJ \cap \BC_{\BI}$ divides $\mathcal D$ into two simply connected domains whose boundaries are simple closed Jordan curves. Let $L$ be now a closed bounded line segment  starting from $s_0,$ not being tangent to $\cJ,$ and such that $L\cap \sigma(A)=\{s_0\}.$ Consider $\mathcal D^\prime$ an auxiliar disk in the slice $\BC_{\BI}$ which tangent to $\tilde\cJ\cap \BC_{\BI}$ at $s_0$ and it is contained in the subdomain of $\mathcal D$ with meets $L.$ Denote by $\cC^\prime$ the boundary of $\mathcal D^\prime.$

We now define the function $\delta$ on $\rho(A)$ via
$$
\delta(s):=\delta_{k, \BI}(Q_S(A)\{A+(1-s_0)\mathcal{I},B\})
$$
where $\{A,B\}$ denotes the anti-commutator of $A$ and $B$, i.e. $\{A,B\}=AB+BA$ and $\delta_{k, \BI}$ is as in Definition \ref{Def:5.4}. Then $\delta$ is continuous in $s$ on $\rho(A)$ and the operator $Q_S(A)\{A+(1-s_0)\mathcal{I},B\}$ belongs to the same Schatten class as $B$. Furthermore, we have $$
Q_S(T)=Q_S(A)(1+Q_S(A)\{A+(1-s_0)\mathcal{I},B\})^{-1}.
$$
In particular, this means that
$$
-1\notin \sigma_S(Q_S(A)(\{A+(1-s_0)\mathcal{I},B\})
$$
 and the absolute convergence of $\delta_{k, \BI}$ (Lemma~\ref{unkown}, (i)) implies $\delta(s)\neq 0.$

Therefore, we have
$$
Q_S(T)=(\delta(s)^{-1})Q_S(A)\delta(s) \left(1+Q_S(A)\{A+(1-s_0)\mathcal{I},B\} \right)^{-1}.
$$
Since we have to estimate $\|Q_S(T)\|$ we proceed with the estimates for each term in the above equation individually. Because $A$ is normal we have
$$
\|Q_S(A)\|\leq \frac{1}{{\rm dist}([s],\sigma_S(A))^2}\leq \frac{1}{{\rm dist}([s],\cC^\prime)^2}.
$$
Furthermore, we have
$$
|\delta(s)|\leq \exp (K_1\|Q_S(A)\{A+(1-s_0)\mathcal{I},B\}\|_k^k)\leq \exp(K_2\|B\|_k^k \|Q_S(A)\|^k)
$$
with $$K_2=2K_1\|A+(1-s_0)\mathcal{I},B\|.$$
This leads to
$$
\delta(s)\leq \left(\frac{K_2}{{\rm dist}([s],\cC^\prime)^{2k}}\right)
$$
Now, since $|\delta(s)|$ is continuous and strictly positive there exists an analytic function $\alpha$ with $|\delta(s)|=\exp(\mathrm{Re}\alpha(s))$ and $\mathrm{Re}\alpha(s)\leq K_2/d(s,C^\prime)^{2k}$.  Therefore, there exists a constant $K_3$ such that
$$
|\alpha(s)| \leq K_3 |s-z_0|^{-2k-2}
$$
 for $s\in \mathcal D^\prime\cap L$. Hence,
 $$\mathrm{Re}\alpha(s)\geq -K_3|s-z_0|^{-2k-2}$$
  and, consequently, we obtain
$$
\left| \delta(s)^{-1} \right| \leq \exp (K_3|s-z_0|^{-2k-2}).
$$
Furthermore, from Lemma~\ref{unkown}, part (iv) we get
\begin{eqnarray*}
\|\delta(s)(1+Q_S(A)\{A+(1-s_0)\mathcal{I},B\})^{-1}\| & \leq & \exp (K_4 \|Q_S(A)\{A+(1-s_0)\mathcal{I},B\})\|_k^k)\\
& \leq & \exp \left(K_5\frac{\|B\|_k^k}{{\rm dist}([s],\sigma_S(A))^{2k}}\right)
\end{eqnarray*}
which leads to
\begin{eqnarray*}
\|Q_S(T)\| & = & |(\delta(s)^-1)| \|Q_S(A)\| \|\delta(s)(1+Q_S(A)\{A+(1-s_0)\mathcal{I},B\})^{-1}\|\\
& \leq & \exp (K_3|s-s_0|^{-2k-2}) \frac{1}{{\rm dist}([s],\sigma_S(A))^2} \left(K_5\frac{\|B\|_k^k}{{\rm dist}([s],\sigma_S(A))^{2k}}\right).
\end{eqnarray*}
Now, since $L$ is not tangent to $\tilde\cJ \cap \BC_{\BI}$ the term $|s-s_0|/{\rm dist}([s],\sigma_S(A))$ is bounded for $s\in L$. Consequently, there exists a constant $K$ such that
$$
\|Q_S(T)\|\leq \exp (K |s-s_0|^{-2k-2})
$$
 for $s\in L\cap \mathcal D^\prime$.
\end{proof}

\begin{theorem}\label{Th:6.3} If $T = A+B,$ where $A$ is normal and $B \in S_p(J)$ for some $p \geq 1,$ and if $\sigma_S(A) \cap \BC_\BI$ contains an exposed arc for all $\BI \in \BS,$ then $T$ has a non-trivial hyperinvariant subspace.
\end{theorem}
\begin{proof} Let $\tilde\cJ \cap \BC_\BI$ be the exposed arc for $\BI \in \BS.$ By Weyl's theorem \ref{Th:Weyl}
$$
\sigma_S(A) \subset \sigma_S(T) \cup \Pi_{0,S}(A).
$$
 Due to the fact that $A$ is normal on a separable space we have that $\Pi_{0,S}(A)$ is countable and thus, a dense subset of $\tilde\cJ$ is contained in $\sigma_S(T)$ that is to say, $\tilde\cJ\subset \sigma_S(T).$ This further implies that $T$ is not a multiple of the identity operator. Hence, if $\Pi_{0,S}(T) \not= \emptyset$ then $T$ has non-trivial hyperinvariant subspaces.

Suppose now that $\Pi_{0,S}(T) = \emptyset.$ Again by Weyl's theorem \ref{Th:Weyl} $\sigma_S(T) \subset \sigma_S(A)$ and $\tilde\cJ \cap \BC_\BI$ is an exposed arc of $\sigma_S(T) \cap \BC_\BI$ for all $\BI \in \BS.$ Now, by Lemma \ref{Lem:6.2} the $S$-resolvent set of $T$ satisfies the growth condition in Theorem \ref{Th:6.1} with $k=\lceil p \rceil.$ From this the result follows.
\end{proof}

\begin{corollary} \label{Cor:6.4} If $T = A+B,$ where $A$ is normal, $B \in S_p(J)$ for some $p \geq 1,$ $\sigma_S(T)$ contains more than one point, and for all $\BI \in \BS$ we have $\sigma_S(A) \cap \BC_\BI$ contained in a smooth Jordan arc, then $T$ has a non-trivial hyperinvariant subspace.
\end{corollary}
\begin{proof} Let $\tilde\cJ$ be such that $\tilde\cJ \cap \BC_\BI$ is a smooth arc and $\sigma_S(A) \subset \tilde\cJ.$ As in the above one can assume $\Pi_{0,S}(T) \not=\emptyset$ and so $\sigma_S(T) \subset \tilde\cJ.$ If $\sigma_S(T)$ is disconnected the result follows from the Riesz Decomposition Theorem. Hence we can assume that $\sigma_S(T) = \tilde\cJ'$ where $\tilde\cJ' \cap \BC_\BI$ is a non-trivial subarc of $\tilde\cJ \cap \BC_\BI,$ for all $\BI \in \BS.$ Again by Weyl's theorem \ref{Th:Weyl} this result is a consequence of Theorem \ref{Th:6.3}.
\end{proof}

\begin{corollary} \label{Cor:6.5}  If $T = A+B,$ where $A$ is normal, $B \in S_p(J)$ for some $p \geq 1,$  and for all $\BI \in \BS$ we have that $\sigma_S(A) \cap \BC_\BI$ contained in a smooth Jordan arc, then $T$ has a non-trivial hyperinvariant subspace.
\end{corollary}
\begin{proof} By Corollary \ref{Cor:6.4} we can assume that $\sigma_S(T)$ contains only one point. By an appropriated translation we can assume $\sigma_S(T) = \{ 0 \}$ in which case $T$ is a compact operator. By Weyl's theorem \ref{Th:Weyl} we have that $\sigma_S(A) \subset \{ 0 \} \cup \Pi_{0,S}(A).$ If $A$ is not compact, then either $A$ has a non-zero eigenvalue of infinite multiplicity or the eigenvalues of $A$ have an accumulation point at some non-zero value. In either case, there exits $s \not= 0$ and an orthonormal set $\{ x_n \}$ such that $s_n \rightarrow s$ and $(A^2-2\Re(s_n) A + |s_n|^2) x_n =0$ for all $n \in \BN.$
Since $B$ is compact, $Bx_n \rightarrow 0$ and it follows that
\begin{gather*}
\left\| \left( (A+B)^2-2\Re(s) (A+B) + |s|^2 \right) x_n \right\| \\
= \left\| \left( A^2-2\Re(s) A + |s|^2 \right) x_n  + ( \{ A, B \} + B^2 -2\Re(s) B )x_n \right\| \rightarrow 0,
\end{gather*} and so $s \in \Pi_{S}(T).$ This contradicts $\sigma_S(T) = \{ 0 \}$ and thus, $A$ is compact as well as $T=A+B.$ Hence, since the operator $T$ is compact it has a non-trivial invariant subspace.
\end{proof}

\begin{corollary}\label{Cor:6.6} If $T-T^\ast \in S_p(J),$ for some $p \geq 1$ then $T$ has a non-trivial invariant subspace.
\end{corollary}
\begin{proof} The result follows from Corollary \ref{Cor:6.5} since $$2T =(T+T^\ast)+ (T-T^\ast)$$ and $T+T^\ast$ is a Hermitean operator so its $S$-spectrum is real valued.
\end{proof}

\begin{corollary}\label{Cor:6.7} If $\sigma_S(T)$ contains more than one point, and $\mathcal{I} - T^\ast T \in S_p(J)$ for some $p \geq 1$ then $T$ has a non-trivial hyperinvariant subspace.
\end{corollary}
\begin{proof} If either $T$ or $T^\ast$ has non-empty point $S$-spectrum then the result is trivially true.
 Assume now that both, $\Pi_{0,S}(T)$ and $\Pi_{0,S}(T^\ast)$, are empty.
 Now recall that the polar decomposition $T=UP$ for quaternionic operators holds.
 The partial isometry arising from the polar decomposition of $T$  is unitary,
 that is to say, $T=UP$ where $U$ is unitary and $P$ is positive. So we have $P^2=T^+T$, but $\mathcal{I}-P^2$ is a compact Hermitian operator so $P$ is diagonable.
 We suppose that for some orthonormal  basis $\{v_n\}_{n\in \mathbb{N}}$  $Pv_v=v_np_n$, the right eigenvalues are also the point $S$-spectrum and we have
 $$
 \sum_{n\in\mathbb{N}}|1-p_n^2|^p<\infty
 $$
 but we also have
 $$
 \sum_{n\in\mathbb{N}}|1-p_n^2|^p=\sum_{n\in\mathbb{N}}|1-p_n|^p|1+p_n|^p\geq \sum_{n\in\mathbb{N}}|1-p_n|^p.
 $$
 this means that $\mathcal{I}-P\in S_p(J)$ and that
 $$
 T=UP=U(\mathcal{I}-(\mathcal{I}-P))=U-U(\mathcal{I}-P).
 $$
 Since the $S$-spectrum of a unitary operator is contained in the unit sphere of the quaternions and $U(\mathcal{I}-P)\in S_p(J)$ the result follows from the Corollary
  \end{proof}
\begin{corollary}\label{Cor:6.8}
Let $\mathcal{I}-T^*T \in S_p(J)$ for some $p\geq 1$, then $T$ has a non-trivial invariant subspace.
\end{corollary}
\begin{proof}
It follows from Corollary \ref{Cor:6.5}.
\end{proof}

\subsection*{Acknowledgments}  The work of the first and third authors was partially supported by Portuguese funds through the CIDMA - Center for Research and Development in Mathematics and Applications, and the Portuguese Foundation for Science and Technology (``FCT--Funda\c{c}\~ao para a Ci\^encia e a Tecnologia''),  within project UID/MAT/ 0416/2013.



\begin{thebibliography}{99}
\bibitem{adler}
S.~L. Adler, \emph{Quaternionic quantum mechanics and noncommutative dynamics},
  Proceedings of the {S}econd {I}nternational {A}. {D}. {S}akharov {C}onference
  on {P}hysics ({M}oscow, 1996), World Sci. Publ., River Edge, NJ, 1997,
  pp.~337--341.



\bibitem{FUCGEN}
D. Alpay, F. Colombo, J. Gantner, D. P. Kimsey,
{\em Functions of the infinitesimal generator of a strongly continuous quaternionic group},
Anal. Appl. (Singap.), {\bf 15} (2017), 279--311.

\bibitem{acgs}
D.~Alpay, F.~Colombo, J.~Gantner, and I.~Sabadini, \emph{A new resolvent
  equation for the {$S$}-functional calculus}, J. Geom. Anal. \textbf{25}
  (2015), 1939--1968.

\bibitem{acks3}
D.~Alpay, F.~Colombo, D.~P. Kimsey, and I.~Sabadini, \emph{The spectral theorem
  for unitary operators based on the $s$-spectrum},  Milan J. Math.,
   {\bf 84} (2016), 41--61.



\bibitem{acsbook}
D.~Alpay, F.~Colombo, and I.~Sabadini, \emph{ Slice hyperholomorphic Schur analysis}.
Operator Theory: Advances and Applications, 256. Birkh\"auser/Springer, Cham, 2016. xii+362.


\bibitem{perturbation}
D.~Alpay, F.~Colombo, and I.~Sabadini, \emph{Perturbation of the generator of a quaternionic evolution
  operator}, Anal. Appl. (Singap.) \textbf{13} (2015), no.~4, 347--370.




\bibitem{ack}
D.~{Alpay}, F.~{Colombo},  D. P. Kimsey,
{\em The spectral theorem for for quaternionic unbounded normal operators based on the $S$-spectrum},
J. Math. Phys. {\bf 57} (2016), 023503

\bibitem{Hinfty}
 D. Alpay, F. Colombo, T. Qian, I. Sabadini,
 {\em The $H^\infty$functional calculus based on the S-spectrum for quaternionic operators and for n-tuples of noncommuting operators},
 J. Funct. Anal., {\bf 271} (2016), 1544--1584.



\bibitem{BvN}
G.~Birkhoff and J.~von Neumann, \emph{The logic of quantum mechanics}, Ann. of
  Math. (2) \textbf{37} (1936), no.~4, 823--843. \MR{1503312}

\bibitem{Brodskii}
 M. S. Brodskii,
 \emph{Triangular and Jordan representations of linear operators}.
  Translated from the Russian by J. M. Danskin. Translations of Mathematical Monographs, Vol. 32.
   American Mathematical Society, Providence, R.I., 1971. viii +246 pp.



 \bibitem{FJTAMS} F. Colombo, J. Gantner, {\em
Fractional  powers of quaternionic operators and Kato's formula using slice hyperholomorphicity},
Transaction of the American Mathematical Society, doi{10.1090/tran/7013}.

\bibitem{OurNewPaper}
F. Colombo, J. Gantner, {\em An application of the $S$-functional calculus to fractional diffusion processes},
Preprint 2017.

\bibitem{CGJ}
F. Colombo, J. Gantner, T. Janssens,
{\em The Schatten class  and the Berezin transform for quaternionic operators},
Math. Methods  Appl. Sci., {\bf 39} (2016), 5582--5606.


\bibitem{evolution}
F.~Colombo and I.~Sabadini, \emph{The quaternionic evolution operator}, Adv. Math. \textbf{227}
  (2011), no.~5, 1772--1805.

\bibitem{spectrum}
F.~Colombo and I.~Sabadini, \emph{On some notions of spectra for quaternionic operators and for n -tuples of operators}, C. R. Math. Acad. Sci. Paris,  {\bf 350}  (2012), 399--402.

\bibitem{CSS}
F.~Colombo, I.~Sabadini, and D.~C. Struppa, \emph{Noncommutative functional
  calculus}, Progress in Mathematics, vol. 289, Birkh\"auser/Springer Basel AG,
  Basel, 2011, Theory and applications of slice hyperholomorphic functions.
  \MR{2752913}

\bibitem{CSS2016} F. Colombo, I. Sabadini, D. Struppa, \textit{Entire slice regular functions}.
  SpringerBriefs in Mathematics. Springer, Cham, 2016. v+118.


\bibitem{DS1958}
N.~Dunford and J.~T. Schwartz,
\emph{Linear operators. {P}art {I}}, Wiley
  Classics Library, John Wiley \& Sons, Inc., New York, 1988, General theory,
  With the assistance of William G. Bade and Robert G. Bartle, Reprint of the
  1958 original, A Wiley-Interscience Publication.

\bibitem{DS1963}
N.~Dunford and J.~T. Schwartz,
\emph{Linear operators. {P}art {II}}, Wiley Classics Library, John
  Wiley \& Sons, Inc., New York, 1988, Spectral theory. Selfadjoint operators
  in Hilbert space, With the assistance of William G. Bade and Robert G.
  Bartle, Reprint of the 1963 original, A Wiley-Interscience Publication.


\bibitem{12}
G.~Emch, \emph{M\'ecanique quantique quaternionienne et relativit\'e
  restreinte. {I}}, Helv. Phys. Acta \textbf{36} (1963), 739--769.

\bibitem{14}
D.~Finkelstein, J.~M. Jauch, S.~Schiminovich, and D.~Speiser, \emph{Foundations
  of quaternion quantum mechanics}, J. Mathematical Phys. \textbf{3} (1962),
  207--220.

\bibitem{Ruzhansky} V. Fischer, M. Ruzhansky, Quantization on nilpotent Lie groups, Birkh\"auser, 2016.

\bibitem{DA}
J. Gantner, {\em A direct approach to the $S$-functional calculus for closed operators},
J. Operator Theory,  {\bf 77:2} (2017), 101--145.

\bibitem{GMP}
R.~Ghiloni, V.~Moretti, and A.~Perotti, \emph{Continuous slice functional
  calculus in quaternionic {H}ilbert spaces}, Rev. Math. Phys. \textbf{25}
  (2013), no.~4, 1350006, 83.

\bibitem{spectcomp}
R.~Ghiloni, V.~Moretti, and A.~Perotti,
 \emph{Spectral properties of compact normal quaternionic operators},
  Hypercomplex Analysis: New Perspectives and Applications (Swanhild Bernstein,
  Uwe Kaehler, Irene Sabadini, and Frank Sommen, eds.), Trends in Mathematics,
  Springer International Publishing, 2014, pp.~133--143 (English).

\bibitem{GK1}
I. C. Gohberg,  M.G. Krein,  \emph{Theory and applications of Volterra operators in Hilbert space}. Translated from the Russian by A. Feinstein. Translations of Mathematical Monographs, Vol. 24 American Mathematical Society, Providence, R.I. 1970 x+430 pp.

\bibitem{GK2}
I. C. Gohberg,  M.G. Krein,  \emph{
 Introduction to the theory of linear nonselfadjoint operators}.
  Translated from the Russian by A. Feinstein. Translations of Mathematical Monographs,
  Vol. 18 American Mathematical Society, Providence, R.I. 1969.

  \bibitem{21}
L.~P. Horwitz and L.~C. Biedenharn, \emph{Quaternion quantum mechanics: second
  quantization and gauge fields}, Ann. Physics \textbf{157} (1984), no.~2,
  432--488. \MR{768240 (86h:81053a)}

  \bibitem{jefferies} B. Jefferies, {\em Spectral properties of noncommuting operators},
Lecture Notes in Mathematics, 1843, Springer-Verlag, Berlin, 2004.


\bibitem{jmc} B. Jefferies, A. McIntosh, {\em The Weyl calculus and Clifford analysis},
Bull. Austral. Math. Soc., {\bf 57} (1998), 329--341.

\bibitem{jmcpw} B. Jefferies, A. McIntosh, J. Picton-Warlow, {\em The monogenic functional
calculus}, Studia Math., {\bf 136} (1999), 99--119.



\bibitem{Kato} T. Kato, {\em Perturbation theory for linear operators}. Reprint of the 1980 edition. Classics in Mathematics. Springer-Verlag, Berlin, 1995.

\bibitem{KM} A. Kriegl, P. W. Michor, {\em Differentiable perturbation of unbounded operators}, Math. Ann.  {\bf 327}  (2003),  191--201.

\bibitem{KMR}
A. Kriegl, P. W. Michor, A. Rainer, {\em Many parameter H\"older perturbation of unbounded operators}, Math. Ann.  {\bf 353}  (2012),  519--522.


\bibitem{Livsic}
M. S. Livsic,
\emph{On spectral decomposition of linear nonself-adjoint operators}. (Russian)
Mat. Sbornik N.S. 34 (76), (1954). 145--199. English transl., Amer. Math. Soc. Transl. (2) 5 (1957), 67--114.

  \bibitem{mcp} A. McIntosh, A. Pryde, {\em A functional calculus for several commuting
operators}, Indiana U. Math. J., {\bf 36} (1987), 421--439.

\bibitem{Stein}
A. Nagel, E.M. Stein, Lectures on Pseudo-Differential Operators: Regularity Theorems and Applications to Non-Elliptic Problems, Princeton University Press, 1979

\bibitem{RR1973}
H. Radjavi, P. Rosenthal, \textit{Invariant subspaces}, Springer-Verlag Berlin Heidelberg New York, 1973.


\bibitem{S2012}
K.~Schm{\"u}dgen, \emph{Unbounded self-adjoint operators on {H}ilbert space},
  Graduate Texts in Mathematics, vol. 265, Springer, Dordrecht, 2012.
  \MR{2953553}

\bibitem{SzNFBK}
 B. Sz.-Nagy, ; C. Foias, H. Bercovici,  L. Kerchy,
 \emph{Harmonic analysis of operators on Hilbert space}.
  Second edition. Revised and enlarged edition. Universitext. Springer, New York, 2010. xiv+474

\bibitem{Viswanath}
K.~Viswanath, {\em Normal operations on quaternionic Hilbert spaces}, Trans.
  Amer. Math. Soc. {\bf162} (1971), 337--350.

\bibitem{Teichmueller}
O. Teichm{\"u}ller, {\em Operatoren im {W}achsschen {R}aum}, J. Reine
  Angew. Math. {\bf 174} (1936), 73--124.


\end{thebibliography}
\end{document}